\documentclass{article}
\usepackage[utf8]{inputenc}
\usepackage[T1]{fontenc}
\usepackage{lmodern}
\usepackage{graphicx}
\usepackage{multicol}
\usepackage[top=2.5cm, bottom=2.5cm, left=2.5cm , right=2.5cm]{geometry}

\usepackage{pgf,tikz,pgfplots}
\pgfplotsset{compat=1.14}
\usepackage{mathrsfs}
\usetikzlibrary{arrows}
\usepackage{hyperref}
\usepackage{amsmath,amssymb,amsthm}
\usepackage{amsthm}

\newtheorem{theorem}{Theorem}[section]
\newtheorem*{theorem*}{Theorem}
\newtheorem{proposition}[theorem]{Proposition}
\newtheorem{lemma}[theorem]{Lemma}
\newtheorem{corollary}[theorem]{Corollary}
\newtheorem{example}[theorem]{Example}

\theoremstyle{definition}
\newtheorem{definition}[theorem]{Definition}

\theoremstyle{remark}
\newtheorem{remark}[theorem]{Remark}

\newcounter{exocmpt}[section]
\newcounter{partcmpt}[exocmpt]
\newcounter{questioncmpt}[exocmpt]
\newcounter{subquestioncmpt}[questioncmpt]

\theoremstyle{definition}

\newcommand{\cp}{\mathbb{C}\mathbb{P}}
\newcommand{\h}{\mathbb{H}}
\newcommand{\s}{\mathbb{S}}
\newcommand{\proj}{\mathbb{P}}
\newcommand{\rp}{\mathbb{R}\mathbb{P}}
\newcommand{\cmplx}{\mathbb{C}}
\newcommand{\real}{\mathbb{R}}

\newcommand{\pinteger}{\mathbb{N}}

\newcommand{\ante}{^{-1}}

\newcommand{\etoile}{^{\ast}}

\newcommand{\sep}{\;|\;}

\newcommand{\class}{\mathcal{C}}

\newcommand{\im}{\text{im}\,}
\newcommand{\sing}{\text{Sing}\,}

\newcommand{\regind}{_{\text{reg}}}
\newcommand{\anclosure}{^{\,\text{an}}}

\newcommand{\gr}{\text{Gr}}

\newcommand{\inftyline}{\overline{\cmplx}_{\infty}}

\begin{document}

\title{On projective billiards with open subsets of triangular orbits}
\author{Corentin Fierobe}
\date\today
\maketitle


\newcommand{\ptrp}{\mathbb{P}(T\rp^2)}
\newcommand{\ptrpd}{\mathbb{P}(T\rp^d)}
\newcommand{\ptcp}{\mathbb{P}(T\cp^2)}
\newcommand{\bref}{\mathbb{P}(T\cp^2)}
\newcommand{\brefd}{\mathbb{P}(T\cp^d)}
\newcommand{\az}{\mbox{az}}
\newcommand{\negl}{\mbox{o}}
\newcommand{\huit}{\frac{1}{8}}

\def\mcc{\mathcal C}
\def\rp{\mathbb{RP}}
\def\mcf{\mathcal F}
\def\cc{\mathbb C}
\def\oc{\overline{\cc}}
\def\oci{\oc_{\infty}}
\def\cp{\mathbb{CP}}
\def\wt#1{\widetilde#1}
\def\rr{\mathbb R}
\def\var{\varepsilon}
\def\mce{\mathcal E}
\def\mct{\mathcal T}
\def\la{\lambda}

\begin{abstract}
Ivrii's Conjecture states that in every billiard in Euclidean space the set of periodic orbits has measure zero. It implies that for every $k\geq2$ there are no k-reflective billiards, i.e., billiards having an open set of k-periodic orbits. This conjecture is open in Euclidean spaces, with just few partial results. It is known that in the two-dimensional sphere there exist 3-reflective billiards (Yu.M.Baryshnikov). All the 3-reflective spherical billiards were classified in a paper by V.Blumen, K.Kim, J.Nance, V.Zharnitsky: the boundary of each of them lies in three orthogonal big circles. In the present paper we study the analogue of Ivrii's Conjecture for projective billiards introduced by S.Tabachnikov. In two dimensions there exists a 3-reflective projective billiard, the so-called right-spherical billiard, which is the projection of a spherical 3-reflective billiard. We show that the only 3-reflective planar projective billiard with piecewise smooth boundary is the above-mentioned right-spherical billiard. In higher dimensions, we prove the non-existence of 3-reflective projective billiards with piecewise smooth boundary, and also the non-existence of projective billiards with piecewise smooth boundary having a subset of triangular orbits of non-zero measure in the phase space.
\end{abstract}

\tableofcontents


\section{Introduction}

Ivrii's conjecture states that for any billiard with smooth boundary in a Euclidean space, the set of periodic orbits has measure zero. To prove this conjecture, it is enough to show that for any $k\in\pinteger$, the set of periodic orbit of period $k$ has measure $0$ (and in particular, has empty interior). We say that a billiard is $k$-reflective, if its set of $k$-periodic orbits has non-empty interior. For planar billiards, this means the existence of a two-dimensional family of $k$-periodic orbits. The proof of the conjecture was made in the case of a billiard with a regular analytic convex boundary, see \cite{vasiliev}. The conjecture is also true for $3$-periodic orbits, and was proved in \cite{bary, rychlik, stojanov, vorobets, wojtkowski}. For $4$-periodic orbits, it was proven in \cite{glutkud1, glutkud2} and a complete classification of $4$-reflective complex analytic billiards was presented in \cite{glut}.

Ivrii's conjecture was also studied in manifolds of constant curvature: it was proven to be true for $k=3$ in the hyperbolic plane $\h^2$, see \cite{VKNZ}; the case of the sphere $\s^2$ was apparently firstly studied in \cite{bary_introuvable}, as quoted in \cite{VKNZ} but we were not able to find the correponding paper. The sphere is in fact an example of space were Ivrii's conjecture is not true, see \cite{VKNZ}, which contains a classification of all the 3-reflective billiards on the sphere.

\begin{figure}[!ht]
\centering
\input{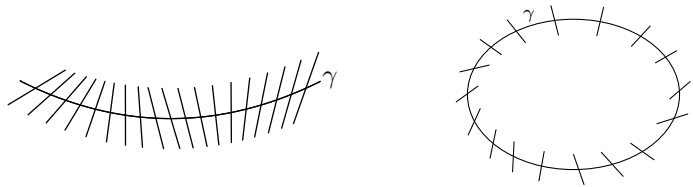}
\hspace{1cm}
\begin{tikzpicture}[line cap=round,line join=round,>=triangle 45,x=1.0cm,y=1.0cm]
\clip(-1.6,-1.2) rectangle (1.6,1.2);
\draw [rotate around={0:(0,0)}] (0,0) ellipse (1.41cm and 1cm);
\draw (-1.56,-0.1)-- (-1.27,0.11);
\draw (-1.51,0.3)-- (-1.12,0.39);
\draw (-1.24,0.73)-- (-0.96,0.53);
\draw (-0.88,0.99)-- (-0.64,0.69);
\draw (-0.3,1.18)-- (-0.2,0.79);
\draw (0.4,1.16)-- (0.31,0.78);
\draw (1.02,0.91)-- (0.77,0.64);
\draw (1.37,0.58)-- (1.09,0.42);
\draw (1.55,0.17)-- (1.28,-0.06);
\draw (1.07,-0.87)-- (0.76,-0.65);
\draw (0.66,-1.09)-- (0.37,-0.77);
\draw (0.14,-1.2)-- (0,-0.8);
\draw (-0.46,-1.15)-- (-0.39,-0.76);
\draw (-0.81,-1.03)-- (-0.79,-0.63);
\draw (-1.11,-0.85)-- (-1.03,-0.47);
\draw (-1.4,-0.54)-- (-1.24,-0.19);
\draw (1.11,-0.39)-- (1.52,-0.26);
\begin{scriptsize}
\draw[color=black] (-0.59,1.04) node {$\gamma$};
\end{scriptsize}
\end{tikzpicture}
\caption{Left: A piece of curve $\gamma$ endowed with a projective field of lines. Right: A convex closed curve $\gamma$ endowed with a projective field of lines.}
\label{fig1}
\end{figure}

In this paper we study a generalization of billiards: the so-called \textit{projective billiards} introduced in \cite{taba_projectif, taba_projectif_ball}. A {\it projective billiard} is a domain 
$\Omega\subset\rr^n$  with boundary endowed with a transverse line field $L$, called 
{\it the projective field of lines} (see Figure 1). The  reflection in planar projective billiard with a 
boundary curve $\gamma$ acts on lines 
according to the following reflection law.  Given a point $A\in\gamma$, let $L(A)$ denote the corresponding 
line of the field. Two lines $\ell$ and $\ell^*$ through $A$ are {\it symmetric} with respect to the pair 
($L(A)$, $T_A\gamma$), 
if the four lines $L(A)$, 
$T_A\gamma$, $\ell$, $\ell^*$ (identified with the corresponding 1-subspaces in $T_A\rr^2$) 
form a harmonic tuple: their cross-ratio is equal to 1. See Section 2 for more details. 

In higher dimensions the notion of 
symmetry with respect to $L(A)$  is analogous. Namely, two lines $\ell,\ell^*\subset T_A\rr^n$ 
are {\it symmetric with respect to the pair} ($L(A), T_A\partial\Omega$), if  $\ell$ ,$\ell^*$, $L(A)$ lie in the same 2-subspace 
$\Pi\subset T_A\rr^n$ and the four lines $L(A)$, $T(A):=\Pi\cap T_A\partial\Omega$, $\ell$, $\ell^*$ in $\Pi\simeq\rr^2$ 
form a harmonic tuple.
\begin{remark} 
The above harmonicity condition 
is equivalent to the condition that  there exists a linear isomorphism $H:T_A\rr^2\to \rr^2$ (respectively, 
$\Pi\to\rr^2$)  
sending $L(A)$ and $T_A\gamma$ to orthogonal lines and $\ell$, $\ell^*$ to symmetric lines with respect to 
$H(T_A\gamma)$ (or equivalently, with respect to $H(L(A))$). 
\end{remark}

 The definition of orbits for  projective billiards is the same as in the usual billiard dynamic. 
 
\begin{definition}
A {\it  $k$-orbit} in a convex projective billiard $\Omega$ 
with transversal line field $L$ on $\partial\Omega$ is a sequence of vertices $A_1,\dots,A_k\in\partial\Omega$ 
such that for every $j=2,\dots,k-1$ one has $A_j\neq A_{j\pm1}$, the lines $A_{j-1}A_j$, $A_jA_{j-1}$ 
are not tangent to $\partial\Omega$ at $A_j$  and are symmetric with respect to the pair ($L(A_j)$, $T_{A_j}\partial\Omega$). A $k$-orbit is {\it periodic,} if 
the above statements also hold  for $j=1,k$ with $A_0=A_k$, $A_{k+1}=A_1$. 
In the  general case, when $\Omega$ is not necessarily convex, a $k$-orbit ($k$-periodic orbit)
 is defined as above with the additional 
 condition that each edge $A_jA_{j+1}$ lies in $\Omega$ (except for its endpoints). 
 \end{definition}
 
 \begin{example} A billiard $\Omega\subset\rr^n$ in Euclidean space can be viewed 
 as a projective billiard, with the boundary $\partial\Omega$ being equipped with 
 the normal line field. The reflection law and orbits in  the Euclidean billiard and in the corresponding 
 projective billiard are the same.
 \end{example} 

Furthermore, as described in [14], examples of projective billiards can be obtained from metrics 
which are {\it projectively equivalent} to the Euclidean one, meaning that their geodesics are  
straight lines.
Indeed for a given metric $g$ on $\rr^n$ projectively equivalent to the Euclidean one 
and any hypersurface $\Gamma\subset\rr^n$ one can define the field of $g$-orthogonal lines to 
$\Gamma$. The projective billiard thus constructed is equivalent to the billiard on $\Gamma$ 
with the reflection acting on geodesics as the reflection in the metric $g$, (see [14], Example 7).

\begin{example} Metrics of constant curvature on space forms are projectively equivalent to 
the Euclidean metric. Indeed each complete simply connected $n$-dimensional space $\Sigma$ of non-zero 
constant curvature, i.e., a space form (sphere or hyperbolic space) is realized as the unit sphere 
(half-pseudosphere of radius -1) 
in a Euclidean (respectively, Minkowski) space $\rr^{n+1}$ (after rescaling the metric by constant factor). 
The tautological projection  $\Sigma\to\rp^n$  sends the geodesics to straight lines and thus, sends the metric 
on $\Sigma$ to a metric projectively equivalent to the Euclidean one. Consider a billiard  
 $\Omega\subset\Sigma$  with reflection acting on geodesics and  defined by the metric on $\Sigma$. 
Its boundary equipped with  the normal line field is projected to a projective billiard in $\rp^n$. 
 The billiard orbits in $\Omega$ are projected to  orbits of the latter projective billiard. 
\end{example}

\begin{definition}
A projective billiard is said to be {\it $k$-reflective,} if the set of its $k$-periodic orbits has a non-empty interior (meaning that it has a $2(n-1)$-dimensional family of $k$-periodic orbits, where $n$ is the dimension of the ambient space.). 
\end{definition}

As was already mentioned above,  the  version of Ivrii's conjecture for triangular orbits in billiards on 2-sphere is false. 
The following example represents  a 3-reflective spherical billiard given in [4] and describes its 
tautological projection, which is  a 3-reflective projective planar billiard. 

\begin{example} Cut the sphere into  8 equal parts by choosing 3 pairwise orthogonal great circles. 
Consider one of these parts, $E$, which is a  geodesic triangle with all angles being right. 
It is a 3-reflective spherical billiard, and moreover, all its orbits are 3-periodic, as was shown in [3, 4]. 
Now let us equip the great circles containing 
its sides  with normal line fields. Their tautological projections to $\rp^2$ form a triple of lines intersecting at three non-collinear points 
$P$, $Q$, $R$ and   equipped with the following fields of lines (i.e., projective billiard structure).
\end{example}

\begin{definition} Consider three non-collinear points $P$, $Q$, $R$ in the projective plane. For every point $M$  
in the line $PQ$ let 
$L(M)\subset T_M\rp^2$ be the line $MR$ through the opposite vertex $R$. The definition of line fields 
on the other lines $QR$, $RP$ is analogous. Let us now choose an affine chart $\rr^2$ containing $P$, $R$, $Q$.  
This yields a  triangle $PQR\subset\rr^2$ equipped  with a projective billiard 
structure, which will be called  the {\it right-spherical billiard}. In what follows the triple of lines 
$PQ$, $QR$, $RP$ equipped with the above line fields will be also called the right-spherical billiard.
\end{definition}
\begin{remark} All right-spherical billiards are projectively isomorphic by definition. The right-spherical billiard is 3-reflective, being the image of the above 3-reflective billiard on sphere under the tautological projection, which sends orbits to orbits. 
\end{remark}

\begin{remark}
We will give another proof of 3-reflectivity of the right-spherical billiard, found by Simon Allais (Proposition 3.3). If the reader is interested, we give in \cite{fierobe1} examples of $2n$-reflective projective billiards constructed inside polygons.
\end{remark}

Let $\Omega$ be a $k$-reflective projective billiard, and let $A_1\dots A_k$ be its $k$-periodic orbit lying in the interior of the set of $k$-periodic orbits. This means that for every $A_1'\in\partial\Omega$ and 
$A_2'\in\partial\Omega$ close to $A_1$ and $A_2$ respectively the edge $A_1'A_2'$ extends to a $k$-periodic orbit 
$A_1'\dots A_k'$ close to $A_1\dots A_k$. The latter statement depends only on {\it germs} of the boundary $\partial\Omega$ at $A_1,\dots,A_k$. This motivates the following definitions.

Consider the fiber bundle $\mathbb P(T\rr^n)$ which can be seen as the set of
pairs $(A,L)$ where $A$ is a point of $\rr^n$ and $L$ is a line through $A$,
with its natural projection $\pi:(A,L)\mapsto A\in\rr^n$.

\begin{definition} A {\it line-framed planar curve} is a  regularly embedded connected curve $\alpha
\subset\mathbb P(T\rr^2)$ with the following properties:

1) The projection $\pi$ sends $\alpha$ diffeomorphically to a regularly embedded curve $a\subset\rr^2$, 
which will be identified with $\alpha$ and called the {\it classical boundary} of the curve $\alpha$.

2) For every $(A,L)\in\alpha$ the line $L$ is transversal  to $T_Aa$. 

A {\it line-framed hypersurface} is a regularly embedded connected $(n-1)$-dimensional  
surface $\alpha\subset\mathbb P(T\rr^n)$
satisfying the above properties 1) and 2), with $a$ being a hypersurface. 

A {\it line-framed  complex analytic planar curve (hypersurface)} is a regularly embedded connected holomorphic curve 
($(n-1)$-dimensional surface) in 
$\mathbb P(T\cc^2)$ (in $\mathbb P(T\cc^n)$) satisfying the above conditions 1) and 2). 
\end{definition}

\begin{figure}[h]
\centering
\input{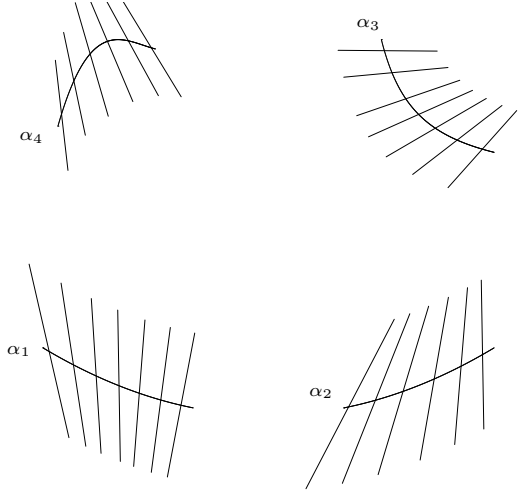}
\caption{A local projective billiard $\mathcal{B}=(\alpha_1,\alpha_2,\alpha_3,\alpha_4)$: each $\alpha_i$ is in $\bref$ and is projected on $a_i$ in $\cp^2$.}
\label{fig5}
\end{figure}

\begin{definition}[A generalization of definitions 1.5, 5.3 in \cite{glut} for projective billiards] A {\it local projective billiard} is a collection of germs of $k$ line-framed hypersurfaces 
$\alpha_1,\dots,\alpha_k\subset\mathbb P(T\rr^n)$; let $a_j$ denote their classical boundaries; see Figure \ref{fig5}. 
The billiard is called {\it smooth (analytic)}, if so are $\alpha_j$. The notion of complex-analytic local 
projective billiard is the same, with $\alpha_j$ being line-framed complex planar curves (hypersurfaces) 
in $\mathbb P(T\cc^n)$. 
 A {\it $k$-periodic orbit} of a local projective billiard is 
a $k$-tuple $(A_1,L_1)\dots (A_k,L_k)$, $(A_j,L_j)\in\alpha_j$, such that for every $j=1,\dots,k$ one has $A_j\neq A_{j\pm1}$ and the lines $A_{j-1}A_j$, $A_jA_{j+1}$ are not tangent to $a_j$ at $A_j$ and 
are  symmetric with respect to the pair $(L_j,T_{A_j}(a_j))$; 
here $A_0=A_k$, $A_{k+1}=A_1$. A local projective billiard is called {\it $k$-reflective,} if the base points 
$A_j$ of the classical boundaries $a_j$ form a $k$-periodic orbit and every pair $(A_1',A_2')\in \alpha_1\times\alpha_2$ 
close to $(A_1,A_2)$ extends to a $k$-periodic orbit $A_1'\dots A_k'$ close to $A_1\dots A_k$. When previous definition is true only for pairs $(A_1',A_2')$ lying in a subset of $\alpha_1\times\alpha_2$ of positive Lebesgue measure, we say that the local projective billiard is {\it $k$-pseudo-reflective}.

\end{definition}

\begin{example} Let $\Omega$ be a $k$-reflective projective billiard. Let $A_1\dots A_k$ be its $k$-periodic orbit. 
Then the germs at $A_j$ of its boundary equipped with the corresponding line fields form a $k$-reflective 
local projective billiard. 
\end{example}
A {\bf version of Ivrii's conjecture}  for {\bf projective billiards} 
is to {\it classify all the $k$-reflective local projective billiards.} 
In the present paper we solve it for 3-reflective local (real and complex) analytic billiards in any dimension.

\begin{theorem} 
\label{thm:main_theorem}
The local 3-reflective (real) $\class^{\infty}$-smooth planar billiards are the (real) right-spherical billiards. 
\end{theorem}

\begin{theorem} 
\label{thm:main_theorem_multidim}
There are no 3-pseudo-reflective local $\class^{\infty}$-smooth (real) projective billiards in $\rr^d$ with $d\geq3$.
\end{theorem}

In the paper we will first prove their analytic versions, and extend them to the $\class^{\infty}$-smooth case:

\begin{theorem} 
\label{thm:main_theorem_analytic}
The local 3-reflective real (complex) analytic planar billiards are the real (complex) 
right-spherical billiards. 
\end{theorem}

\begin{theorem} 
\label{thm:main_theorem_multidim_analytic}
There are no 3-reflective local analytic real (complex) projective billiards in $\rr^d$ ($\cc^d$) with $d\geq3$.
\end{theorem}

\begin{remark}
\label{remark:introducing_contruction_pseudo}
An important corollary of the existence of $3$-reflective planar projective billiards is the existence of planar projective billiards which are not $3$-reflective but whose set of triangular orbits has non-zero measure in the space of initial conditions. See Remark \ref{remark:construction_pseudo} for further details. This is not true in the class of classical billiards, see \cite{VKNZ, rychlik, stojanov}.
\end{remark}

\noindent \textbf{Plan of the article.} We first give precise definitions of the projective reflection law (Section \ref{sec:reflection_law}). Then we prove Theorem \ref{thm:main_theorem} in the case when the classical boundaries of the projective billiard are supported by lines (Propositions \ref{prop:reflectivity_and_lines1} and \ref{prop:reflectiv_spherical} in Section \ref{sec:germ_line}). 

After that, we prove Theorem \ref{thm:main_theorem_analytic} in the following way: we suppose that we are given a $3$-reflective local analytic projective billiard $\mathcal{B}=(\alpha,\beta,\gamma)$ such that one of the classical boundaries, say the classical boundary $a$ of $\alpha$, is not a line. We study a certain analytic distribution, called Birkhoff's distribution: it is constructed so that open sets of triangular orbits of 3-reflective projective billiards yield its two-dimensional integral surfaces (see Section \ref{sec:birkhoff_distribution}). We use methods of analytic geometry (like Remmert's proper mapping theorem) to prove that this distribution is of dimension $2$ on the analytic closure $M$ of the set of triangular orbits of $\mathcal{B}$ (Proposition \ref{prop:dimension_birkhoff_ditrib}); and this allows us to conclude that it is integrable on $M$ (Corollary \ref{cor:birkhoff_integrable}). Finally, we exhibit a particular integral surface of Birkhoff's distribution (Subsection \ref{subsec:existence_integral}) corresponding to a certain projective billiard $\mathcal{B}_0=(\alpha,\beta_0,\gamma_0)$ where the classical boundary $b_0$ of $\beta_0$ intersects a certain tangent line $L$ of $a$. But such projective billiard cannot be $3$-reflective, as we show in Section \ref{subsec:integral_not_reflective}, using asymptotic comparisons of complex angles (or azimuths) of orbits whose vertices converge to $L$.

Then we prove Theorem \ref{thm:main_theorem_multidim_analytic} in Section \ref{sec:multidim}, using Theorem \ref{thm:main_theorem_analytic} and the fact that any triangular orbit has its vertices in the same plane.

Finally we extend Theorems \ref{thm:main_theorem_analytic} and \ref{thm:main_theorem_multidim_analytic} to Theorems \ref{thm:main_theorem} and \ref{thm:main_theorem_multidim} using the theorem of Cartan-Kuranishi-Rachevskii on prolongation of Pfaffian systems applied to spaces of jets of germ of surfaces, as described in Section \ref{sec:smooth_case}. More precisely, we show that the $r$-jets of a germ of a smooth integral surface of a Pfaffian system can be approximated by $r$-jets of germs of analytic integral surfaces as close as possible, see Proposition \ref{prop:analytic_approx}. This together with Theorems \ref{thm:main_theorem_analytic} and \ref{thm:main_theorem_multidim_analytic} can be used to prove Theorems \ref{thm:main_theorem} and \ref{thm:main_theorem_multidim}.

\section{Projective law of reflection}
\label{sec:reflection_law}

We first recall some definitions about the cross-ratio of four lines in $\cp^2$, we introduce harmonic quadruples of lines, and then give a formula for the projective symmetry of lines.

\begin{definition}
Fix any set of coordinates on $\cp^1$. Given three distinct points $X_2$, $X_3$, $X_4$ of $\cp^1$ and a fourth point $X_1$, the cross-ratio $(X_1,X_2;X_3,X_4)$ is defined to be the value $h(X_1)$ where $h:\cp^1\to\cp^1$ is the only projective transformation sending $X_2$ to $1$, $X_3$ to $0$ and $X_4$ to $\infty$.
\end{definition}

\begin{remark}
The cross-ratio is a projective invariant in the sense that for any projective transformation $f:\cp^1\to\cp^1$, it satisfies $(f(X_1),f(X_2);f(X_3),f(X_4))=(X_1,X_2;X_3,X_4)$, and hence doesn't depend on the initial choice of coordinates of $\cp^1$.
\end{remark}

\begin{definition}
\label{def:cross_ratio}
The \textit{cross-ratio of a quadruple of four distinct lines} $(\ell_1,\ell_2,\ell_3,\ell_4)$ intersecting at a point $O$, is the cross-ratio of the quadruple formed by their intersection points $X_1$, $X_2$, $X_3$, $X_4$ with a fifth line $\ell$ not containing $O$, taken in the same order :
$$(\ell_1,\ell_2;\ell_3,\ell_4) = (X_1,X_2;X_3,X_4).$$
It does not depend on the choice of $\ell$.
\end{definition}

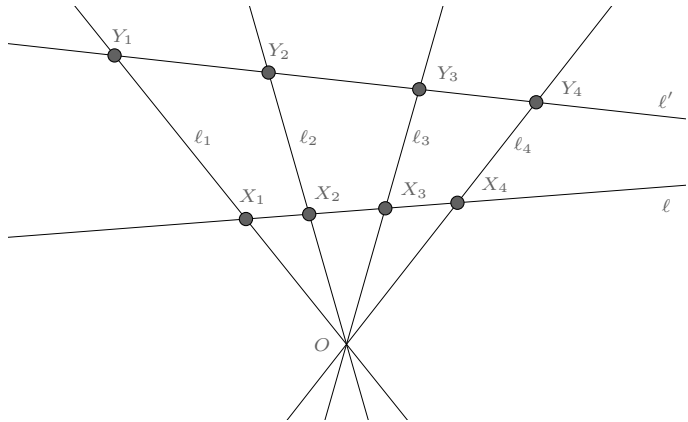
\begin{figure}[h]
\centering
\definecolor{black}{rgb}{0.08235294117647059,0.396078431372549,0.7529411764705882}
\definecolor{black}{rgb}{0.3803921568627451,0.3803921568627451,0.3803921568627451}
\begin{tikzpicture}[line cap=round,line join=round,>=triangle 45,x=0.5cm,y=0.5cm]
\clip(-9,-2) rectangle (9,9);
\draw [line width=0.3pt,domain=-13.57:15.89] plot(\x,{(--46--1*\x)/13});
\draw [line width=0.3pt,domain=-13.57:15.89] plot(\x,{(--126-2*\x)/18});
\draw [line width=0.3pt,domain=-13.57:15.89] plot(\x,{(-0--3.332705882352941*\x)/-2.674823529411764});
\draw [line width=0.3pt,domain=-13.57:15.89] plot(\x,{(-0--3.4621176470588235*\x)/-0.9924705882352942});
\draw [line width=0.3pt,domain=-13.57:15.89] plot(\x,{(-0--3.6177647058823528*\x)/1.0309411764705878});
\draw [line width=0.3pt,domain=-13.57:15.89] plot(\x,{(-0--3.765294117647059*\x)/2.948823529411765});
\begin{scriptsize}
\draw[color=black] (-0.65,0) node {$O$};
\draw [fill=black] (-2.674823529411764,3.332705882352941) circle (2.5pt);
\draw[color=black] (-2.5,3.92) node {$X_1$};
\draw [fill=black] (-0.9924705882352942,3.4621176470588235) circle (2.5pt);
\draw[color=black] (-0.5,4) node {$X_2$};
\draw [fill=black] (1.0309411764705878,3.6177647058823528) circle (2.5pt);
\draw[color=black] (1.8,4.1) node {$X_3$};
\draw [fill=black] (2.948823529411765,3.765294117647059) circle (2.5pt);
\draw[color=black] (3.95,4.25) node {$X_4$};
\draw [fill=black] (-6.1682571399042265,7.685361904433803) circle (2.5pt);
\draw[color=black] (-5.95,8.18) node {$Y_1$};
\draw [fill=black] (-2.0726787719955073,7.230297641332833) circle (2.5pt);
\draw[color=black] (-1.79,7.8) node {$Y_2$};
\draw [fill=black] (1.9335427742870943,6.7851619139681) circle (2.5pt);
\draw[color=black] (2.7,7.2) node {$Y_3$};
\draw [fill=black] (5.043259557344064,6.4396378269617705) circle (2.5pt);
\draw[color=black] (6,6.8) node {$Y_4$};
\draw[color=black] (-3.8,5.5) node {$\ell_1$};
\draw[color=black] (-1,5.5) node {$\ell_2$};
\draw[color=black] (2,5.5) node {$\ell_3$};
\draw[color=black] (4.7,5.3) node {$\ell_4$};
\draw[color=black] (8.5,3.7) node {$\ell$};
\draw[color=black] (8.5,6.5) node {$\ell'$};
\end{scriptsize}
\end{tikzpicture}
\caption{$(\ell_1,\ell_2;\ell_3,\ell_4) = (X_1,X_2;X_3,X_4) = (Y_1,Y_2;Y_3,Y_4).$ The definition of the cross-ratio does not depend on the chosen line $\ell$.}
\label{fig3}
\end{figure}

\noindent We say that a quadruple $(\ell_1,\ell_2,\ell_3,\ell_4)$ of distinct lines is \textit{harmonic} if its cross-ratio is $-1$ : 
$$(\ell_1,\ell_2;\ell_3,\ell_4) = -1.$$
If now we fix a set of coordinates on the line $\ell$ of Definition \ref{def:cross_ratio} such that each intersection point $X_i$ of $\ell_i$ with $L$ can be identified with $z_i\in\cp^1$, then $(\ell_1,\ell_2,\ell_3,\ell_4)$ is harmonic if and only if
\begin{equation}
\label{eq:harmonic_formula}
z_4 = \frac{(z_1+z_2)z_3-2z_1z_2}{2z_3-(z_1+z_2)}.
\end{equation}
Using Equation \eqref{eq:harmonic_formula}, we define the symmetry of lines at a point $A\in\cp^2$ : 

\begin{definition}
Let $L$, $T$ be two distinct lines through $A$. Two lines $(\ell,\ell')$ are said to be symmetric with respect to $(L,T)$ if either $\ell=\ell'=T$, or $\ell=\ell'=L$, or $(L,T,\ell,\ell')$ is a harmonic quadruple of distinct lines.
\end{definition}

\begin{remark}
\label{remark:symmetric_cross_ratio}
Note that the three conditions are coherent with Formula \eqref{eq:harmonic_formula}. It can also be easily shown that the cross-ratio $(L,T,\ell,\ell')$ is invariant by permutations of the lines $L$ and $T$, or of the lines $\ell$ and $\ell'$ or of the pairs $(L,T)$ and $(\ell,\ell')$.
\end{remark}

Finally let us introduce the azimuth of a line. We consider a set of coordinates on $\cp^2$ such that $\cp^2$ can be seen as the disjoint union of $\cmplx^2$ and a complex line, called \textit{infinity line} and denoted by $\inftyline$. A \textit{finite} (respectively, an \textit{infinite}) point of $\cp^2$ will be a point in $\cmplx^2$ (respectively, in $\inftyline$).

\begin{definition}
A complex line $d$ of $\cp^2$ ditinct from the infinity line intersects $\inftyline$ at a unique point called the \textit{azimuth} of $d$, and written $\az(d)$.
\end{definition}

Now fix two lines $L$ and $T$ in $\cp^2$, distinct from $\inftyline$ and intersecting at a finite point $A$. The set $A^{\ast}$ of all lines passing through $A$ can be identified with $\cp^1$, via the map 
\begin{equation}
\label{eq:id_symmetry}
d\in A^{\ast}\mapsto\az(d).
\end{equation}

\begin{proposition}
\label{prop:symmetry}
Via the identification \eqref{eq:id_symmetry}, the symmetry of lines with respect to $(L,T)$ is the unique nontrivial complex conformal involution of $A^{\ast}$, which has $L$ and $T$ as fixed points. It is of the form
\begin{equation}
\label{eq:form_symmetry}
z\mapsto\frac{(\ell+t)z-2\ell t}{2z-(\ell+t)}
\end{equation}
where $\ell$ and $t$ are the azimuths of $L$ and $T$ respectively.
\end{proposition}

\begin{remark}
\label{rem:degenerate_reflection}
Notice that when $\ell=t$, the map of Equation \eqref{eq:form_symmetry} is constant to $\ell$. This shows that when both lines $L$ and $T$ are the same, the reflection law degenerates and we can say that any two lines $(d,d')$ are symmetric with respect to $(L,L)$ if and only if $d=L$ or $d'=L$.
\end{remark}

\section{Proof of Theorem \ref{thm:main_theorem_analytic} when $a$, $b$, $c$ are supported by lines}
\label{sec:germ_line}

Let $(\alpha,\beta,\gamma)$ be a 3-reflective local projective billiard. Let $a$, $b$, $c$ denote the classical boundaries of the curves $\alpha$, $\beta$, $\gamma$.

\vspace{0.2cm}
In this section we prove Theorem \ref{thm:main_theorem_analytic} when $a$, $b$, $c$ are supported by lines. Let us first define right-spherical billiards properly with the tools we introduced on projective billiards. Take a line $\ell\subset\cp^2$ and a point $P\notin \ell$. One can define a line-framed curve $\omega(\ell,P)\subset\bref$ (which is in fact algebraic) as the set of $(A,L)$ where $A$ is on $\ell$ and $L$ is the line $AP$ :
$$\omega(\ell,P)=\{(A,AP)\in\bref\sep A\in\ell\}.$$

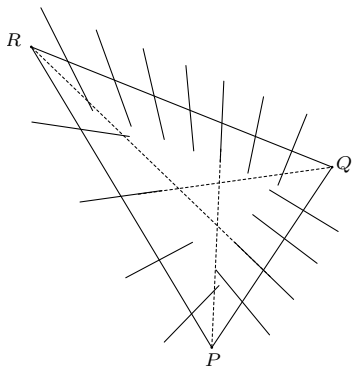
\begin{figure}[h]
\centering
\begin{tikzpicture}[line cap=round,line join=round,>=triangle 45,x=0.8cm,y=0.8cm]
\clip(-3.4,-2.5) rectangle (2.4,4);
\draw (2,1)-- (0,-2);
\draw (0,-2)-- (-3,3);
\draw (-3,3)-- (2,1);
\draw (-2.84,3.65)-- (-1.98,1.94);
\draw (-1.92,3.28)-- (-1.34,1.68);
\draw (-1.14,2.97)-- (-0.79,1.46);
\draw (-0.43,2.69)-- (-0.3,1.27);
\draw (0.2,2.43)-- (0.14,1.09);
\draw (0.86,2.17)-- (0.6,0.9);
\draw (1.58,1.88)-- (1.1,0.7);
\draw (-2.99,1.75)-- (-1.37,1.51);
\draw (-2.19,0.42)-- (-0.83,0.61);
\draw (-1.43,-0.84)-- (-0.32,-0.25);
\draw (-0.79,-1.91)-- (0.12,-0.97);
\draw (0.07,-0.71)-- (0.96,-1.79);
\draw (1.36,-1.2)-- (0.38,-0.25);
\draw (0.68,0.21)-- (1.75,-0.6);
\draw (0.96,0.62)-- (2.1,-0.07);
\draw [line width=0.4pt,dash pattern=on 1pt off 1pt] (0.97,-0.82)-- (-3,3);
\draw [line width=0.4pt,dash pattern=on 1pt off 1pt] (-1.22,0.55)-- (2,1);
\draw [line width=0.4pt,dash pattern=on 1pt off 1pt] (0.15,1.3)-- (0,-2);
\begin{scriptsize}
\fill [color=black] (2,1) circle (0.5pt);
\draw[color=black] (2.2,1.04) node {$Q$};
\fill [color=black] (0,-2) circle (0.5pt);
\draw[color=black] (0.02,-2.2) node {$P$};
\fill [color=black] (-3,3) circle (0.5pt);
\draw[color=black] (-3.3,3.12) node {$R$};
\end{scriptsize}
\end{tikzpicture}
\caption{The right-spherical billiard $\mathcal{B}(P,Q,R)$}
\label{fig:spherical}
\end{figure}

\begin{definition}
\label{def:spherical_billiard}
We call any projective billiard $\mathcal{B}$ \textit{right-spherical} if there are three points $P$, $Q$, $R$ not on the same line, such that
$$\mathcal{B}=\left(\omega(PQ,R),\omega(QR,P),\omega(RP,Q)\right),$$
see Figure \ref{fig:spherical}. We can write $\mathcal{B} = \mathcal{B}(P,Q,R)$.
\end{definition}

\begin{remark}
As explained in the introduction, the name right-spherical billiard comes from the construction which let us understand billiards on $\s^2$ as projective billiards. The construction is the following : consider the projection 
$$\pi:\s^2\subset\real^3\to\rp^2.$$
It sends the geodesics of $\s^2$ (which are its great circles) onto the lines of $\rp^2$. If $b$ is a smooth curve in $\s^2$ and $M\in b$, the transverse line at $\pi(M)\in\pi(b)$ is given by the projection of the orthogonal geodesics to $T_Mb$. Note then that when you project and then complexify the $3$-reflective billiards on $\s^2$ given in \cite{VKNZ}, you get right-spherical billiards.
\end{remark}

\begin{proposition}
\label{prop:reflectiv_spherical}
Any right-spherical billiard is $3$-reflective.
\end{proposition}

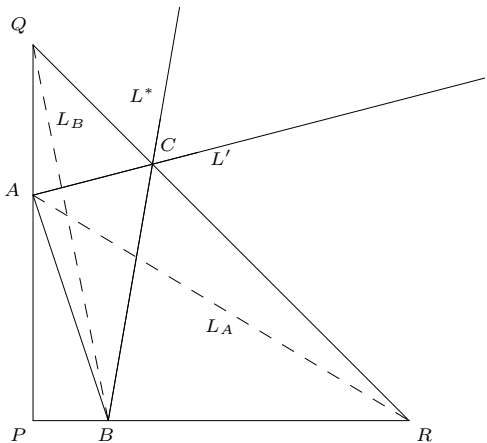
\begin{figure}[h]
\centering
\definecolor{qqqqff}{rgb}{0,0,0}
\definecolor{rvwvcq}{rgb}{0,0,0}
\definecolor{wrwrwr}{rgb}{0,0,0}
\begin{tikzpicture}[line cap=round,line join=round,>=triangle 45,x=5cm,y=5cm]
\clip(-0.5,-0.1) rectangle (1.2,1.1);
\draw [line width=0.3pt] (0,0)-- (0,1);
\draw [line width=0.3pt] (0,0)-- (1,0);
\draw [line width=0.3pt] (1,0)-- (0,1);
\draw [line width=0.3pt,dash pattern=on 5pt off 5pt,color=qqqqff] (0,0.6)-- (1,0);
\draw [line width=0.3pt,dash pattern=on 5pt off 5pt,color=qqqqff] (0.2,0)-- (0,1);
\draw [line width=0.3pt] (0,0.6)-- (0.2,0);
\draw [line width=0.3pt,domain=0.2:2.51110225791798] plot(\x,{(-0.1364762157619824--0.6823810788099121*\x)/0.11761892119008793});
\draw [line width=0.3pt,domain=0:2.51110225791798] plot(\x,{(--0.19057135271405276--0.08238107880991208*\x)/0.31761892119008794});
\draw [line width=0.3pt] (0.2,0)-- (0.3385248449141526,0.8036694450013475);
\draw [line width=0.3pt] (0,0.6)-- (0.4392656677239148,0.7139326947389109);
\begin{scriptsize}
\draw[color=black] (-0.03942603816603807,-0.03744588010741597) node {$P$};
\draw[color=black] (-0.037416165592291786,1.0519050548630626) node {$Q$};
\draw[color=black] (1.0398755339357098,-0.03744588010741597) node {$R$};
\draw[color=black] (-0.05550501875600823,0.6137528337863756) node {$A$};
\draw[color=qqqqff] (0.49720993902421634,0.25) node {$L_A$};
\draw[color=black] (0.1937191803885294,-0.03744588010741597) node {$B$};
\draw[color=qqqqff] (0.1,0.8) node {$L_B$};
\draw[color=black] (0.36,0.7343451882111518) node {$C$};
\draw[color=black] (0.29,0.8669967780784057) node {$L\etoile$};
\draw[color=black] (0.5,0.7) node {$L'$};
\end{scriptsize}
\end{tikzpicture}
\caption{The line $AB$ is reflected into the lines $L'$ and $L\etoile$ in $A$ and $B$ respectively. Their intersection point is $C$, which in fact lies on $QR$ by a simple computation.}
\label{fig4}
\end{figure}

\begin{proof}
This proof was found by Simon Allais in a talk we had about harmonicity in a projective space. Let $\mathcal{B} = \mathcal{B}(P,Q,R)$ be a right-spherical billiard. Let $A\in PQ$ and $B\in QR$ such that $A\neq B$.

Let $C\in RP$ be such that $AB$, $BC$, $QR$, $BP$ are harmonic lines. Define $C'\in RP$ similarly: $AB$, $AC''$, $PQ$, $AR$ are harmonic lines.

Let us first show that necessarily $C=C'$.
Consider the line $RP$ and let $K$ be its point of intersection with $AB$. Let us consider harmonic quadruples of points on $RP$. By harmonicity of the previous defined lines passing through $B$, the quadruple of points $(K,C,R,P)$ is harmonic. Doing the same with the lines passing through $A$, the quadruple of points $(K,C',R,P)$ is harmonic. Hence $C=C'$ since the projective transformation defining the cross-ratio is one to one.

Now let us prove that the lines $BC$, $AC$, $PR$, $CQ$ through $C$ are harmonic lines. Consider the line $PQ$: $BC$ intersects it at a certain point denoted by $T$, $AC$ at $A$, $PR$ at $P$ and $CQ$ at $Q$. But the quadruple of points $(T,A,P,PQ)$ is harmonic since there is a reflection law at $B$ whose lines intersect $PQ$ exactly in those points.
\end{proof}

\begin{proposition}
\label{prop:reflectivity_and_lines1}
Suppose that $\mathcal{B}=(\alpha,\beta,\gamma)$ is a $3$-reflective complex-analytic local projective billiard such that its classical borders $a$, $b$, $c$ are supported by pairwise distinct lines. Then $\mathcal{B}$ is right-spherical.
\end{proposition}

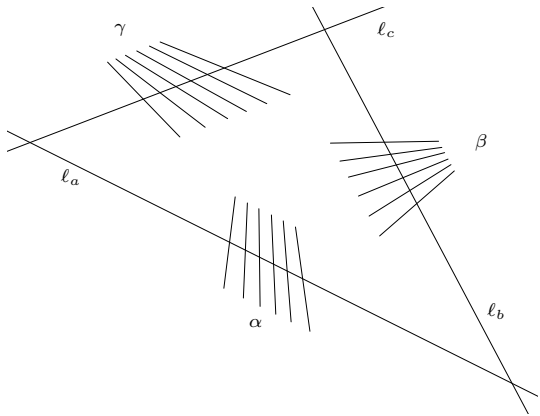
\begin{figure}[h]
\centering
\definecolor{rvwvcq}{rgb}{0,0,0}
\definecolor{qqqqff}{rgb}{0,0,0}
\begin{tikzpicture}[line cap=round,line join=round,>=triangle 45,x=3cm,y=3cm]
\clip(-1.1,-0.8) rectangle (1.25,1);
\draw [line width=0.3pt,domain=-1.654113568915245:3.3932987001391792] plot(\x,{(-0.16-0.8*\x)/1.6});
\draw [line width=0.3pt,domain=-1.654113568915245:3.3932987001391792] plot(\x,{(-1.2277789527356262--1.5691014399127163*\x)/-0.8310198934037676});
\draw [line width=0.3pt,domain=-1.654113568915245:3.3932987001391792] plot(\x,{(--1.0230587210441373--0.5009276323773554*\x)/1.3053277216669545});
\draw [line width=0.3pt,color=qqqqff] (-0.14093264369608133,-0.24637842805425736)-- (-0.090747654511657,0.15896186920455588);
\draw [line width=0.3pt,color=qqqqff] (-0.05446004694753462,-0.2896147264285307)-- (-0.036702281543815274,0.131939182720635);
\draw [line width=0.3pt,color=qqqqff] (0.019659321694076873,-0.32667441074933645)-- (0.014254784397292608,0.10646064975008107);
\draw [line width=0.3pt,color=qqqqff] (0.08914622979558781,-0.3614178648000919)-- (0.06984431087850126,0.07866588650947676);
\draw [line width=0.3pt,color=qqqqff] (0.15708898438373148,-0.39538924209416376)-- (0.12234553033297602,0.052415276782239374);
\draw [line width=0.3pt,color=qqqqff] (0.24047327410554437,-0.4370813869550702)-- (0.17639090330081772,0.025392590298318522);
\draw [line width=0.3pt,color=qqqqff] (0.04921651272513028,0.572389630901921)-- (-0.4706893691568044,0.8277886768078134);
\draw [line width=0.3pt,color=qqqqff] (-0.041087770469984575,0.5377348005297482)-- (-0.5268361892772861,0.8062419853512397);
\draw [line width=0.3pt,color=qqqqff] (-0.5769864660553699,0.786996504701037)-- (-0.12415336436607993,0.5058578622009946);
\draw [line width=0.3pt,color=qqqqff] (-0.6192281026848497,0.7707860138999011)-- (-0.22377749125112056,0.4676264838430335);
\draw [line width=0.3pt,color=qqqqff] (-0.6562124666437625,0.7565930341865296)-- (-0.33588611505757576,0.4246041021374378);
\draw [line width=0.3pt,color=qqqqff] (-0.4233561297280056,0.8459531018531494)-- (0.15241244056143988,0.6119917100916681);
\draw [line width=0.3pt,color=qqqqff] (0.33130561602396125,0.39734460354641676)-- (0.8106343513587597,0.4056321135828327);
\draw [line width=0.3pt,color=qqqqff] (0.3735588899635165,0.3175635134596003)-- (0.8250628437474745,0.37838876073098693);
\draw [line width=0.3pt,color=qqqqff] (0.837327115365644,0.35523181056854314)-- (0.4111925350039675,0.24650503809594904);
\draw [line width=0.3pt,color=qqqqff] (0.8525318240728224,0.326522834387277)-- (0.45562999257754133,0.1625998521693388);
\draw [line width=0.3pt,color=qqqqff] (0.8655523491459822,0.30193795404580465)-- (0.5026297605126291,0.07385660531342031);
\draw [line width=0.3pt,color=qqqqff] (0.8803366427601895,0.27402279030269827)-- (0.5492936900856976,-0.014252523801629502);
\begin{scriptsize}
\draw[color=black] (-0.8180181217670579,0.24605511831000076) node {$\ell_a$};
\draw[color=black] (1.0666232550013968,-0.34675026020080313) node {$\ell_b$};
\draw[color=black] (0.5800431177266321,0.9) node {$\ell_c$};

\draw[color=black] (0,-0.4) node {$\alpha$};
\draw[color=black] (1,0.4) node {$\beta$};
\draw[color=black] (-0.6,0.9) node {$\gamma$};
\end{scriptsize}
\end{tikzpicture}
\caption{The local projective billiard $\mathcal{B}=(\alpha,\beta,\gamma)$ is such that $a$, $b$, $c$ are supported by the lines $\ell_a$, $\ell_b$, $\ell_c$ respectively.}
\label{fig6}
\end{figure}

\begin{proof}
Write $\ell_a$, $\ell_b$, $\ell_c$ the lines which respectively support $a$, $b$, $c$. First note that two lines among $\ell_a$, $\ell_b$, $\ell_c$ cannot be the same (otherwise $3$-periodic orbits cannot exist). There are two cases to consider:

Case 1. The three lines $\ell_a$, $\ell_b$, $\ell_c$ intersect at the same point $R$.

Case 2. The three lines $\ell_a$, $\ell_b$, $\ell_c$ do not intersect at the same point.

\noindent \textbf{Case 1.} We suppose that the three lines $\ell_a$, $\ell_b$, $\ell_c$ intersect at the same point $R$ and we show a contradiction. Choose any $m_A=(A,L_A)$ on $\alpha$ such that $L_A\neq \ell_a$ and $A\neq R$. The line $L_A$ intersects $\ell_c$ at a certain point $R_C\neq R$ and $\ell_b$ at a certain point $R_B\neq R$. Fix $m_B=(B,L_B)$ in $\beta$ and consider $m_C=(C,L_C)$ in $\gamma$ such that $(m_A,m_B,m_C)$ is a triangular orbit. Since the quadruple of lines through $A$, $(\ell_a,L_A;AC,AB)$, is harmonic, so is the quadruple of lines through $B$, $(\ell_b,BR_C;BC,AB)$ (since the intersection points of these lines with $\ell_c$ are the same in each quadruple). We deduce that $L_B=BR_C$. The same arguments can be applied for $m_C$ to show that $L_C=CR_B$. Thus the curves $\beta$ and $\gamma$ coincide with the curves $\{(B',B'R_C) | B'\in \ell_b\}$ and $\{(C',C'R_B) | C'\in\ell_c\}$ near their points $(B,L_B)$ and $(C,L_C)$ respectively. Similarly for $m_A'=(A',L_A')$ in a neighborhood of $m_A$, the line $L_A'$ should also intersect $\ell_b$ at $R_B$ and $\ell_c$ at $R_C$ (by the same arguments). Hence $L_A'=R_BR_C$ and $A=A'$ which is impossible, since $A'\in a$ was arbitrary.

\noindent \textbf{Case 2.} Suppose that the three lines $\ell_a$, $\ell_b$, $\ell_c$ do not intersect at the same point. Choose any $m_A=(A,L_A)$ in $\alpha$ such that $A$ belongs neither to $\ell_b$, nor to $\ell_c$, and denote by $R$ the intersection point of $\ell_b$ with $\ell_c$. Let us show that $L_A=AR$. 

Suppose the contrary: $L_{A}\neq AR$. First we extend the curves $\beta$ and $\gamma$ analytically. For any $B\in\ell_b$, the line $AB$ is reflected at point $m_A$ into a certain line intersecting $\ell_c$ at a point $C(B)$. In particular, when $B=R$, $C(B)\neq B$ otherwise $L_A$ would be the line $AR$. Hence $B\neq C(B)$ for any $B\in\ell_b$ and the map $B\mapsto BC(B)$ is well-defined and analytic on $\ell_b$ (by Equation \eqref{eq:form_symmetry}). Therefore there is an analytic field of lines $L_B:\ell_b\to\cp^{2\ast}$ such that for any $B\in\ell_b$, $(\ell_b,L_B(B);AB,BC(B))$ is a harmonic quadruple of lines through $B$ (see Equation \eqref{eq:form_symmetry}), and $\beta$ extends to the curve $\{(B,L_B(B))\sep B\in\ell_b\}$. We can apply the same arguments to $\gamma$ which extends to a curve $\{(C,L_C(C)\sep C\in\ell_c\}$ where $L_C:\ell_c\to\cp^{2\ast}$ is an analytic field of lines. Notice that the constructed extensions do not depend on the choice of $m_A$ since they are analytic and they contain the initial line-framed curves $\beta$ and $\gamma$ respectively.

Now take $m_B=(B,L_B)\in\beta$ such that $B=R$ and consider $C=C(B)$ and then $m_{C}=(C,L_{C}(C))$ in $\gamma$. As we noticed previously, $C\neq R$ and $BC=\ell_c$. Note that $(m_A,m_B,m_C)$ is a triangular orbit by analyticity of the projective reflection law at $B$ and $C$ (by Equation \ref{eq:form_symmetry}). Therefore, the quadruple of lines through $B$, $(\ell_b,L_B;AR,\ell_c)$, is harmonic, and this is true for any $m_A$ in $\alpha$ which is impossible since the lines $AR$ cannot be the same (this follows from the fact that $\ell_a$ doesn't go through $R$).
\end{proof}

\section{Study of Birkhoff's distribution}
\label{sec:birkhoff_distribution}

\subsection{Singular analytic distributions}

We recall some definitions and properties of singular analytic distributions, which can be found in \cite{glut}.

\begin{definition}[\cite{glut}, Lemma $2.27$]
Let $W$ be a complex manifold, $\Sigma\subset W$ a nowhere dense closed subset, $k\in\{0,\ldots,n\}$ and $\mathcal{D}$ an analytic field of $k$-dimensional planes defined on $W\setminus\Sigma$. We say that $\mathcal{D}$ is a \textit{singular analytic distribution of dimension $k$ and singular set $\sing(\mathcal{D})=\Sigma$} if $\mathcal{D}$ extends analytically to no points in $\Sigma$ and if for all $x\in W$, one can find holomorphic $1$-forms $\alpha_1$,..., $\alpha_p$ defined on a neighborhood $U$ of $x$ and such that for all $y\in U\setminus\Sigma$,
$$\mathcal{D}(y) = \bigcap_{i=1}^p \ker\alpha_i(y).$$
\end{definition}

\noindent Singular analytic distribution can be restricted on analytic subsets:

\begin{proposition}[\cite{glut}, Definition $2.32$]
Let $W$ be a complex manifold, $M$ an irreducible analytic subset of $W$ and $\mathcal{D}$ a singular analytic distibution on $W$ with $M\nsubseteq\sing(\mathcal{D})$. Then there exists an open dense subset $M^o\regind$ of $x\in M\regind$ for which
$$\mathcal{D}_{|M}(x) := \mathcal{D}(x)\cap T_xM$$
has minimal dimension. We say that $\mathcal{D}_{|M}$ is a singular analytic distribution on $M$ of singular set $\sing(\mathcal{D}):=M\setminus M^o\regind$.
\end{proposition}

\begin{remark}
When $M$ is not irreducible anymore, we still can restrict $\mathcal{D}$ to $M$ by looking at its restriction to each of the irreducible components of $M$.
\end{remark}

\noindent As in the smooth case, we can look for integral surfaces defined by the following

\begin{definition}[\cite{glut}, Definition $2.34$]
Let $\mathcal{D}$ be a $k$-dimensional analytic distribution on an irreducible analytic subset $M$ and $\ell\in\{0,\ldots,k\}$. An \textit{integral $\ell$-surface of $\mathcal{D}$} is a submanifold $S\subset M\setminus\sing(\mathcal{D})$ of dimension $\ell$ such that for all $x\in S$, we have the inclusion $T_xS\subset \mathcal{D}(x)$. The analytic distribution $\mathcal{D}$ is said to be \textit{integrable} if each $x\in M\setminus\sing(\mathcal{D})$ is contained in an integral $k$-surface. (In this case the $k$-dimensional integral surfaces form a holomorphic foliation of the manifold $M\setminus\sing(\mathcal D)$, by Frobenius theorem.)
\end{definition}

\noindent We can finally introduce the following lemma, which will be used in a key result (Corollary \ref{cor:birkhoff_integrable}). We recall here that the analytic closure of a subset $A$ of a complex manifold $W$, is the smallest analytic subset of $W$ containing $A$. We denote it by $\overline{A}\anclosure$.

\begin{lemma}[\cite{glut}, Lemma $2.38$]
\label{lemma:sad_frobenius_integrable}
Let $\mathcal{D}$ be a $k$-dimensional singular analytic distribution on an analytic subset $N$ and $S$ be a $k$-dimensional integral surface of $\mathcal{D}$. Then the restriction of $\mathcal{D}$ to $\overline{S}\anclosure$ is an integrable analytic distribution of dimension $k$.
\end{lemma}

\noindent The proof is the same as in \cite{glut}:

\begin{proof}
Write $M=\overline{S}\anclosure$. First, let us prove that $\mathcal{D}_{|M}$ is $k$-dimensional. Consider the subset 
$$A := \{x\in M\setminus\sing(\mathcal{D}_{|M})\sep \mathcal{D}(x)\subset T_xM\}.$$
It contains $S\setminus\sing(\mathcal{D}_{|M})$, hence its closure, which is an analytic subset of $M$, contains $S$. By definition, $\overline{A}\anclosure=M$ which implies that $\mathcal{D}_{|M}$ is $k$-dimensional.

Now let us show that $\mathcal{D}_{|M}$ is integrable. The argument is similar: define the subset $B$ of those $x\in M\setminus\sing(\mathcal{D}_{|M})$ such that the Frobenius integrability condition is satisfied. $B$ contains $S\setminus\sing(\mathcal{D}_{|M})$ and its closure is an analytic subset of $M$ containing $S$, hence it is the whole $M$. Thus Frobenius theorem can be applied on the manifold $M\setminus\sing(\mathcal{D}_{|M})$, which implies the result.
\end{proof}

\subsection{Birkhoff's distribution and the $3$-reflective billiard problem}
\label{subsec:birk_distrib_1}

Let us define Birkhoff's distribution attached to a complex analytic line-framed curve $\alpha$, and establish its link with the local projective billiard (Proposition \ref{prop:link_birkhoff_billiard}). We define the space $\mathcal{P}$ as the fiber bundle 
$$\mathcal{P} = \bref\underset{\cp^2}{\times}\bref$$ 
that is the set of triples $(A,L,T)$ where $A\in\cp^2$ and $L,T$ are lines in $T_A\cp^2$. Consider the space $\alpha\times\mathcal{P}^2$ of triples $z=\left(m_A,m_B,m_C\right)$ where $m_A=(A,L_A)\in\alpha$, $m_B=(B,L_B,T_B)\in\mathcal{P}$ and $m_C=(C,L_C,T_C)\in\mathcal{P}$ these notations will be used all along the paper). We define also a certain number of projections: 
\begin{itemize}
	\item $\pi_{B},\pi_C:\alpha\times\mathcal{P}^2\to\cp^2$ such that $\pi_B(z) = B$ and $\pi_C(z) = C$;
	\item $\pi_{\alpha}:\alpha\times\mathcal{P}^2\to\alpha$ the projection onto $\alpha$;
	\item $\pi_{\beta},\pi_{\gamma}:\alpha\times\mathcal{P}^2\to\bref$ such that $\pi_{\beta}(z)=(B,L_B)$ and $\pi_{\gamma}(z)=(C,L_C)$.
\end{itemize}

\noindent\textbf{Phase space.} In the space $\alpha\times\mathcal{P}^2$, we consider the subspace $M_{\alpha}^{0}$ of triangular billiard orbits having one reflection in $\alpha$, that is the set of triples
$$z=[(A,L_A),(B,L_B,T_B),(C,L_C,T_C)]$$
where $(A,L_A)\in\cp^2$, $(B,L_B,T_B)\in\mathcal{P}$ and $(C,L_C,T_C)\in\mathcal{P}$ with further properties that $A,B,C$ do not lie on the same line, $L_B\neq T_B$, $L_C\neq T_C$, the lines $AB$ and $AC$ are symmetric with respect to $(L_A,T_A)$, the lines $AB$ and $BC$ are symmetric with respect to $(L_B,T_B)$, and the lines $AC$ and $CB$ are symmetric with respect to $(L_C,T_C)$. Take $M_{\alpha}$ to be the analytic closure of $M_{\alpha}^{0}$. The set $M_{\alpha}$ is a smooth analytic subset of $\alpha\times\mathcal{P}^2$ of dimension $6$.

\vspace{0.2cm}\noindent\textbf{Birkhoff's distribution attached to $\alpha$.} We consider the distribution $D$ on $\alpha\times\mathcal{P}^2$ defined for all $z$ by
$$D(z) = d\pi_B^{-1}(T_B)\cap d\pi_C^{-1}(T_C).$$

\begin{definition}
We call \textit{Birkhoff's distribution attached to $\alpha$} the restriction of $D$ to the phase space $M_{\alpha}$, and still denote it by $D$.
\end{definition}

\begin{proposition}
\label{prop:link_birkhoff_billiard}
Let $z_0\in M_{\alpha}^0$ such that one can find a germ of $2$-dimensional integral analytic surface $S$ of $D$ containing $z_0$. Suppose that $d\pi_G$ has rank $1$ on $S$ for $G\in\{\alpha,\beta,\gamma,B,C\}$. Then there exists a $3$-reflective complex-analytic local projective billiard defined by $(\alpha,\pi_{\beta}(U),\pi_{\gamma}(U))$, with $U$ a sufficiently small neighborhood of $z_0$ in $S$.
\end{proposition}

\begin{remark}
The link between billiards and Birkhoff's distribution is not new and was introduced in \cite{bary2}. See \cite{bary2} Proposition 4.1 for a similar result in classical billiards.
\end{remark}

\begin{proof}
By the constant rank theorem, there is a neighborhood $U$ of $z_0$ in $S$ such that $\hat{b}:=\pi_B(U)$ and $\hat{c}:=\pi_C(U)$ are immersed curves of $\cp^2$, and such that $\pi_{\beta}(U)$ and $\pi_{\gamma}(U)$ are immersed curves of $\ptcp$. It follows from the assumptions that the restrictions $\pi:\pi_{\beta}(U)\to \hat{b}$ and $\pi:\pi_{\gamma}(U)\to \hat{c}$ are biholomorphisms. Therefore, the inverse maps of these restrictions, denoted by $\beta$ and $\gamma$ respectively, are line-framed curves.

Now since $S$ is an integral surface of $D$, for $z=(A,L_{A},B,L_{B},T_{B},C,L_{C},T_{C})\in U$ we have $T_B\hat{b} = d\pi_B(T_zS)\subset T_B$ and $T_C\hat{c} = d\pi_C(T_zS)\subset T_C$. Yet these spaces have the same dimension $1$, hence $T_B\hat{b} = T_B$ and $T_C\hat{c}=T_C$. But since $S\subset M_{\alpha}^0$, the lines passing through $B$, $AB$, $BC$, $T_Bb$ and $L_B$ are harmonic, and the same is true for the lines at $C$ and at $A$. Therefore any $z\in U$ corresponds to a periodic orbit
$i(z)=((A,L_{A}),(B,L_{B}),(C,L_{C}))$ of $(\alpha,\beta,\gamma)$. 

The map $\pi_{\alpha,B,C}:U\to \alpha\times \hat{b}\times \hat{c}$ verifying $\pi_{\alpha,B,C}(z) = ((A,L_A),B,C)$ is an immersion, since there is a map $s:\alpha\times \hat{b}\times \hat{c}\to\alpha\times\mathcal{P}^2$ verifying $s\circ\pi_{\alpha,B,C}(z)=z$ for any $z\in U$ (biholomorphicity of the projections $\pi_{\beta}(U)\to\hat b$ and $\pi_{\gamma}(U)\to\hat c$). Hence $\pi_{\alpha,B,C}(U)$ is an immersed surface of $\alpha\times \hat{b}\times \hat{c}$ and projects into a non-empty open subset of $I$ being either $\alpha\times\hat{b}$, or $\hat{b}\times \hat{c}$, or $\hat{c}\times\alpha$. This means that there is an open subset of initial conditions in $I$ corresponding to periodic orbits. By Lemma \ref{lemma:initial_conditions} (stated below), one can suppose that $I\subset\alpha\times\hat{b}$ and the proof is complete.
\end{proof}

\begin{lemma}
\label{lemma:initial_conditions}
Let $(\alpha_1,\alpha_2,\alpha_3)$ be a local projective billiard; let $a_j$ denote their classical boundaries. Let $(A_1,L_1)\in\alpha_1$, $(A_2,L_2)\in\alpha_2$, $(A_3,L_3)\in\alpha_3$ such that $A_1$, $A_2$, $A_3$ are pairwise distinct, the line $A_1A_2$ is transverse to $a_1$ at $A_1$, the line $A_2A_3$ is transverse to $a_3$ at $A_3$ and the reflection law holds at $A_2$. Then one can define a smooth map $T$ in a neighborhood $W\subset\alpha_1\times\alpha_2$ of $((A_1,L_1),(A_2,L_2))$, $T:W\subset \alpha_1\times \alpha_2\to \alpha_2\times \alpha_3$, such that $T$ is of rank $2$ and such that $T((A_1',L_1'),(A_2',L_2'))=((A_2',L_2'),(A_3',L_3'))$ where $A_3'$ is defined by the condition that the lines $A_1'A_2'$ and $A_2'A_3'$ are symmetric with respect to the pair $(L_2', T_{A_2'}a_2)$.
\end{lemma}

\begin{proof}
$T$ is well defined and smooth in a neighborhhod of $((A_1,L_1),(A_2,L_2))$ by the implicit functions theorem and by the transversality conditions. $T$ is of rank $2$ if and only if when $A_2'=A_2$ is fixed, the map $(A_1',L_1')\mapsto(A_3',L_3')$ is of rank $1$. And this is true by a computation using Formula \ref{eq:form_symmetry} which we omit (it uses the transversality condition at $A_1$).
\end{proof}

\subsection{Reduction of the space of orbits}
\label{subsec:reduction_space_orbits}

Let $\mathcal{B}=(\alpha,\beta,\gamma)$ be a $3$-reflective complex-analytic local projective billiard, and suppose that $a=\pi\circ\alpha$ is not supported by a line. We are interested in the problem of finding $2$-dimensional integral surfaces of $D$ in $M_{\alpha}$ (see Subsection \ref{subsec:birk_distrib_1} for precise definitions). In what follows we suppose that $\alpha$ is a line-framed complex analytic (regularly embedded connected) curve, say, bijectively parametrized by unit disk, such that $T_Aa\neq L_A$ for all $(A,L_A)\in\alpha$ (it is always possible by shrinking $\alpha$). 

By assumption we already have a $2$-dimensional integral surface $S$ of $D$ in $M_{\alpha}$ given by the triangular orbits of the $3$-reflective local projective billiard $\mathcal{B}$. Consider
$$M=\overline{S}^{\,\text{an}}$$
the analytic closure of $S$ in $M_{\alpha}$. In this subsection we want to prove that $\dim M\leq 4$.\\

\noindent\textbf{Construction of two analytic subsets containing $M$.} Consider the open subset $\Omega_1$ of $\alpha\times\mathcal{P}$ defined by those
$$(m_A,m_B),\, m_A = (A,L_A),\, m_B = (B,L_B,T_B)$$
for which $B\notin L_A\cup T_Aa$, and $A\notin L_B\cup T_B$ (in particular $A\neq B$). Then for each $(m_A,m_B)\in\Omega_1$, set $AB^{\ast m_A}$ the line obtained by the symmetry of the line $AB$ at $A$ with respect to $(L_A,T_Aa)$, and $AB^{\ast m_B}$ the line obtained by the symmetry of the line $AB$ at $A$ with respect to $(L_B,T_B)$. By construction of $\Omega_1$, the lines $AB^{\ast m_A}$ and $AB^{\ast m_B}$ are distinct and intersect at a point $C$. This defines a map 
$$\hat{\gamma}:\Omega_1\to\cp^2$$
which is analytic by the implicit function theorem. Then, for each $(m_A,m_B)\in\Omega_1$, the map $\hat{\gamma}(\cdot,m_B) : \alpha \to \cp^2$ is analytic and non constant (unless if at least $a$ is a line through $B$). Define 
$\Gamma(m_A,m_B)$ to be the tangent line to the germ at $\hat\gamma(m_A,m_B)$ of the analytic curve $\hat{\gamma}(\cdot,m_B)$. The analytic map
$$\Gamma:\Omega_1\to\bref$$
has the fortunate property that for all $z=[m_A,m_B,(C,L_C,T_C)]\in S$ with $(m_A,m_B)\in\Omega_1$, we have $T_C = \Gamma(m_A,m_B)$ by the $3$-reflective property of the local projective billiard corresponding to $S$. Now, $\Gamma$ extends analytically to $\Gamma':\Omega_1\to\mathcal{P}$
by setting $\Gamma'(m_A,m_B)=(\hat{\gamma}(m_A,m_B),L_C,\Gamma(m_A,m_B))$ where $L_C$ is chosen so that $AB^{\ast m_A}$ and $AB^{\ast m_B}$ are symmetric with respect to $(L_C,\Gamma(m_A,m_B))$ (see Equation \eqref{eq:form_symmetry}). Thus $S$ is in the analytic set $$M_{\alpha,\gamma} = \overline{\{(m_A,m_B,m_C)\in \Omega_1\times\mathcal{P}\sep \Gamma'(m_A,m_B)=m_C\}}$$
defined as the closure of the graph of $\Gamma'$.

Now we can do the same constructions with $m_C = (C,L_C,T_C)$ instead of $m_B$, and define analogously $\Omega_2\subset\alpha\times\mathcal{P}$, $\hat{\beta}:\Omega_2\to\cp^2$, $\mathcal{B}:\Omega_2\to\bref$, $\mathcal{B}':\Omega_2\to\mathcal{P}$ and $M_{\alpha,\beta}$ as the closure of the graph of $\mathcal{B}':\Omega_2\to\mathcal{P}$. We have obviously :

\begin{proposition}
\label{prop:relations_tangents}
$M\subset M_{\alpha}\cap M_{\alpha,\beta}\cap M_{\alpha,\gamma}$.
\end{proposition}

From this, we deduce :

\begin{proposition}
\label{prop:finite_fibers}
The natural projection 
$$\pi:M\to F=\{(m_A,B,C)\sep m_A\in \alpha, B\in\cp^2, C\in AB^{\ast m_A}\}$$
where $AB^{\ast m_A}$ is the line symmetric to $AB$ with respect to $(L_A,T_Aa)$, has generically finite fibers. In more details, the image $\pi(M)$ is an analytic subset in $F$ and there exists a dense subset $U\subset\pi(M)$ (a complement to a proper analytic subset) such that $\pi\ante(y)$ is finite for every $y\in U$. Hence $\dim M\leq 4$.
\end{proposition}

\begin{proof}
Notice first that $F$ is an analytic subset of $\alpha\times\left(\cp^2\right)^2$ of dimension $4$ and $\pi(M)$ is an analytic subset of $F$. Consider the set $U\subset F$ of $(m_A,B,C)\in F$ for which $A,B,C$ do not lie on the same line, $A\in a$ is regular, $L_A\neq T_Aa$, $B$ and $C$ do not lie in $L_A\cup T_Aa$: by analyticity of these conditions, $U$ is an open set which is the complement of a proper analytic subset in $F$.

Take $(m_A,B,C)\in U$. The set $\pi^{-1}(m_A,B,C)$ is an analytic set of $\{m_A\}\times\proj(T_B\cp^2)^2\times\proj(T_C\cp^2)^2$ (which can be identified with $(\cp^1)^4$) hence is algebraic by Chow's theorem. Suppose $\pi^{-1}(m_A,B,C)$ isn't finite. Then at least one of the projections of $\pi^{-1}(m_A,B,C)$ onto either $L_B$, $T_B$, $L_C$ or $T_C$ is infinite. We suppose that it is the projection onto $T_B$, $\pi_{T_B}:\pi^{-1}(m_A,B,C)\to\cp^1$ (the cases of the other projections are treated similarly). By Remmert propper mapping an Chow's theorems, its image is $\cp^1$. Now take 
$$z_1=(m_A,B,L_B,T_B,C,L_C,T_C)\in\pi^{-1}(m_A,B,C)$$
such that $T_B\neq AB$ or $BC$. Since $\im\pi_{T_B}=\cp^1$, one can also find
$$z_2=(m_A,B,L_B',L_B,C,L_C',T_C')\in\pi^{-1}(m_A,B,C),$$
that is such that $T_B$ has been replaced by the previous $L_B$. Now $L_B'$ is such that the lines $AB$ and $BC$ are symmetric with respect to $(L_B',L_B)$ and they are already symmetric with respect to $(L_B,T_B)$, hence $L_B'=T_B$ (by Equation \eqref{eq:form_symmetry}).

By Proposition \ref{prop:relations_tangents}, $(C,L_C,T_C)=\Gamma'(m_A,B,L_B,T_B)$ and $(C,L_C',T_C')=\Gamma'(m_A,B,T_B,L_B)$. But since the maps $\hat{\gamma}(\cdot,(B,L_B,T_B))$ and $\hat{\gamma}(\cdot,(B,T_B,L_B))$ are the same (because the symmetry of lines through $B$ with respect to $(L_B,T_B)$ is the same as the symmetry of lines through $B$ with respect to $(T_B,L_B)$), hence $(C,L_C,T_C)=(C,L_C',T_C')$.

The same argument works with $(C,L_C,T_C)$: one should have $(B,L_B,T_B)=\mathcal{B}'(C,L_C,T_C)$ and $(B,T_B,L_B)=\mathcal{B}'(C,L_C',T_C')=\mathcal{B}'(C,L_C,T_C)$. It follows that $(B,L_B,T_B)=(B,T_B,L_B)$ hence that $L_B=T_B$. But this contradicts the harmonicity of $(AB,BC;T_B,L_B)$ since the lines $AB,BC,T_B$ are pairwise distinct.

Therefore $\pi^{-1}(m_A,B,C)$ is finite as soon as $(m_A,B,C)\in U$, which concludes the proof.
\end{proof}

Now consider a point $m_A\in \alpha$, and denote by $W_{m_A}$ the set of $z=(m_A,\ast,\ast)\in M$. It is an algebraic set by Chow's theorem.

\begin{lemma}
\label{lemma:epimorphic_projections}
If $\dim M\geq 3$, for a generic $m_A\in\alpha$ we have either $\pi_B(W_{m_A}) = \cp^2$ or $\pi_C(W_{m_A}) = \cp^2$.
\end{lemma}

\begin{proof}
For a generic $m_A\in\alpha$, $\dim W_{m_A}\geq 2$, since $dim M\geq3$, and the map $\pi$ of Proposition \ref{prop:finite_fibers} restricts to a map with generically finite fibers on $W_{m_A}$
$$\pi_{|W_{m_A}} : W_{m_A}\to F_{m_A}:=\{(m_A,B,C)\sep B\in\cp^2, C\in AB^{\ast m_A}\}.$$
Fix such a $m_A$ with further condition that $A\notin b$. Now $\pi_B(W_{m_A})$ is an algebraic set (Chow Theorem). It contains $b$, thus is of dimension at least one. Suppose it is of dimension $1$. The map $\pi_{|W_{m_A}}$ has its image in the algebraic set 
$$G_{m_A}:=\{(m_A,B,C)\sep B\in\pi_B(W_{m_A}), C\in AB^{\ast m_A}\}\subset F_{m_A}$$
of dimension $2$. Thus, its restriction to at least one
irreducible component $\widehat W_{m_A}$ of the set $W_{m_A}$ is an epimorphic map onto a two-dimensional irreducible component $\widehat G_{m_A}$ of the set $G_{m_A}$. Hence $\pi_C(\widehat W_{m_A}) = \pi_C(\widehat G_{m_A})$ contains all lines $AB^{\ast m_A}$ with $B\in\pi_B(\widehat W_{m_A})$. Since $A\notin b$, there is an uncountable number of distinct such lines (when $B$ varies on $b$ for example). Hence $\pi_C(W_{m_A}) = \cp^2$.

We have proven that for a generic $m_A\in\alpha$, either $\pi_B(W_{m_A}) = \cp^2$ or $\pi_C(W_{m_A}) = \cp^2$.
\end{proof}

\subsection{Integrability of Birkhoff's distribution on $M$}

Consider the restriction of Birkhoff's distribution $D$ to $M$, which is a singular analytic distribution on $M$, and denote it by $D_{M}$. Let us compute its dimension.

\begin{proposition}
\label{prop:dimension_birkhoff_ditrib}
The singular analytic distribution $D_M$ is $2$-dimensional.
\end{proposition}

\begin{proof}
We obviously have $\dim D_M\geq 2$ since $T_zS\subset D_M(z)$ for $z\in S$, $S$ being two dimensional. By Proposition \ref{prop:finite_fibers}, $2\leq\dim M\leq 4$ and so is $\dim D_M$. Consider two cases : $\dim M=3$ and $\dim M=4$ (the case when $\dim M=2$ being obvious). In both cases, take a regular $z=(m_A,B,L_B,T_B,C,L_C,T_C)\in M$ such that $L_B\neq T_B$, $\dim D_M(z)$ is minimal, the points $A$,$B$,$C$ are not on the same line, and $B$ and $C$ do not lie in $L_A\cup T_Aa$.

\textbf{Case when $\dim M=3$.} We just have to find one vector $U\in T_zM$ which is not in $D_M(z)$. By Lemma \ref{lemma:epimorphic_projections} we can suppose that $m_A$ is such that $\pi_B(W_{m_A}) = \cp^2$. Consider then a path $u:]-\varepsilon,\varepsilon[\to M$ with $\varepsilon>0$, such that $u(0)=z$ and $\pi_B\circ u(t)$ is a path along the line $AB$ with non-zero derivative at $0$. Consider the vector $U=u'(0)\in T_zM$. It has the property that $d\pi_B(U)$ is a vector corresponding to the derivative of $\pi_B\circ u(t)$ in $0$ which is non zero and is directed along the line $AB$. Hence
$d\pi_B(U)\notin T_B$, otherwise $T_B=AB$, $L_B\neq T_B$ by assumption, and hence, $C$ could not lie outside $AB$ by the reflection law. We conclude that $U\notin D_M(z)$.

\textbf{Case when $\dim M=4$.} Let us find two linearly independent vectors $U,V\in T_zM$ such that $D_M(z)$ and the $2$-plane spanned by $(U,V)$ have $0$-dimensional intersection. We can suppose that $m_A$ is such that $\dim W_{m_A} = 3$ (generic condition), $\pi_B(W_{m_A})=\mathbb{CP}^2$, and hence by Proposition \ref{prop:finite_fibers} that the projection
$$\pi_{|W_{m_A}} : W_{m_A}\to F_{m_A}:=\{(m_A,B,C)\sep B\in\cp^2, C\in AB^{\ast m_A}\}$$
is epimorphic (Remmert's propper mapping theorem and since $F_{m_A}$ is a connected smooth complex projective algebraic manifold). Hence we can define $U\in T_zM$ from a path $u:]-\varepsilon,\varepsilon[\to M$ with $\varepsilon>0$, such that $u(0)=z$, $\pi_B\circ u(t)$ is a path along the line $AB$ with nonzero derivative in $0$, and $\pi_C\circ u(t) \equiv C$ is constant. We change the roles of $B$ and $C$ and do the same construction to get a certain $V=v'(0)\in T_zM$. Here $v(t)$ is a path in $M$ such that $\pi_B\circ v(t)\equiv const$ and $\pi_C\circ v(t)\in AB^{*m_A}$. We then check that 
\begin{itemize}
	\item $U$ and $V$ are linearly independent since $d\pi_B(U)\neq 0$ and $d\pi_C(V)\neq 0$ while $d\pi_C(U)=0$ and $d\pi_B(V)=0$.
	\item If $pU+qV\in D_M(z)$ for $p,q\in\cmplx$, then $d\pi_B(pU+qV)\in T_B$ by the definition of $D_M$. Yet $d\pi_B(pU+qV) = pd\pi_B(U)$. Thus $p=0$ since otherwise $AB=T_B$ and we get a contradiction with the fact that $A,B,C$ do not lie on the same line. Similarly, we find that $q=0$ by considering $d\pi_C(pU+qV) = qd\pi_C(V)$.
\end{itemize}
This concludes the proof.
\end{proof}

\noindent By Lemma \ref{lemma:sad_frobenius_integrable}, we have the 

\begin{corollary}
\label{cor:birkhoff_integrable}
The singular analytic distribution $D_M$ is integrable.
\end{corollary}

\section{Analytic case: proof of Theorem \ref{thm:main_theorem_analytic}}
\label{sec:border_lines}

Let $\mathcal{B}=(\alpha,\beta,\gamma)$ be a $3$-reflective complex-analytic local projective billiard, whose classical boundaries are denoted by $a$, $b$, $c$. As described in Subsection \ref{subsec:reduction_space_orbits}, we suppose that $\alpha$ is a line-framed complex analytic (regularly embedded connected) curve, say, bijectively parametrized by unit disk, such that $T_Aa\neq L_A$ for all $(A,L_A)\in\alpha$.  We say that a classical border $g$ is \textit{supported by a line} if $\im g$ is contained in a line of $\cp^2$. In this section we prove the

\begin{proposition}
\label{prop:classical_borders_lines}
The classical borders $a$, $b$, $c$ are supported by lines.
\end{proposition}

Proposition \ref{prop:classical_borders_lines} combined with Proposition \ref{prop:reflectivity_and_lines1} will conclude the proof of Theorem \ref{thm:main_theorem_analytic}. To prove Proposition \ref{prop:classical_borders_lines}, we suppose that one of the classical borders, say $a$, is not supported by a line, and show that a contradiction arises: the first important result is Proposition \ref{prop:existence_flat_orbit} giving the existence of a particular $3$-reflective local projective billiard having $\alpha$ in its boundary. The contradiction comes from asymptotic comparisons of compex angles (or \textit{azimuths} defined in Section \ref{sec:reflection_law}) proved in subsection \ref{subsec:integral_not_reflective}. The following remark will be useful:

\begin{remark}
Let $H:U\to\ptcp$ be an analytic map of a connected Riemann surface $U$ such that $h:=\pi\circ H$ is non constant. Then $dh(x)$ is of rank $1$ for all $x$ lying outside a discrete subset of $U$. A point $p\in\im h$ is said to be \textit{regular} if there exists $z\in h\ante(p)$ for which $dh(z)\neq 0$. By shrinking $U$ around $z$ if needed, one can suppose in this case that $h$ and $H$ are diffeomorphisms on their respective images, and therefore $H(U)$ is a complex-analytic line-framed curve and $h(U)$ is its classical boundary.
\end{remark}

In this section, we will use the following classical statement, concerning duality of analytic curves~(see~ \cite{GH}):

\begin{proposition}
\label{prop:a_not_line}
Suppose $a$ is not supported by a line, and let $P\in\cp^2$. Then $P\in T_Aa$ for at most a countable number of $A\in a$.
\end{proposition}

\subsection{Existence of a particular $3$-reflective local projective billiard}
\label{subsec:existence_integral}

The main result of this subsection is Proposition \ref{prop:existence_flat_orbit}, which shows (under the assumption that $a$ is not supported by a line) the existence of a particular $3$-reflective local projective billiard having $\alpha$ in its boundary. We will prove then in next subsection that the existence of such billiard is impossible.

\vspace{0.2cm}
Given two analytic curves $h_1:U_1\to\cp^2$ and $h_2:U_2\to\cp^2$ defined on Riemann surfaces $U_1$, $U_2$ and two points $p_1\in\im h_1$, $p_2\in\im h_2$, we say that \textit{the germs $(h_1,p_1)$ and $(h_2,p_2)$ coincide} if $p_1=p_2$ and there is an open subset $V$ of $\cp^2$ containing the $p_i$ and for which $\im h_1\cap V = \im h_2\cap V$.

\begin{proposition}
\label{prop:first_step_flat_orbit}
Let $\mathcal{B} = (\alpha,\beta,\gamma)$ be a $3$-reflective complex-analytic local projective billiard such that $a$ is not supported by a line. Then there is a $3$-reflective complex-analytic local projective billiard $\mathcal{B}'=(\alpha,\beta',\gamma')$ with classical borders $a$, $b'$, $c'$, $m_{A_0}=(A_0,L_{A_0})\in\alpha$ and $m_{B_0}=(B_0,L_{B_0})\in\beta'$, such that at least one of the following cases holds : 
\begin{enumerate}
\item $A_0=B_0$ and the germs $(a,A_0)$ and $(b',B_0)$ coincide ;
\item the points $A_0$ and $B_0$ are distinct, and $T_{A_0}a$ intersects $b'$ transversally at $B_0$. See Figure \ref{fig7}.
\end{enumerate}
Furthermore, if $F\subset a$ is a discrete subset, we can choose $A_0\notin F$.
\end{proposition}

\begin{figure}[h]
\centering
\input{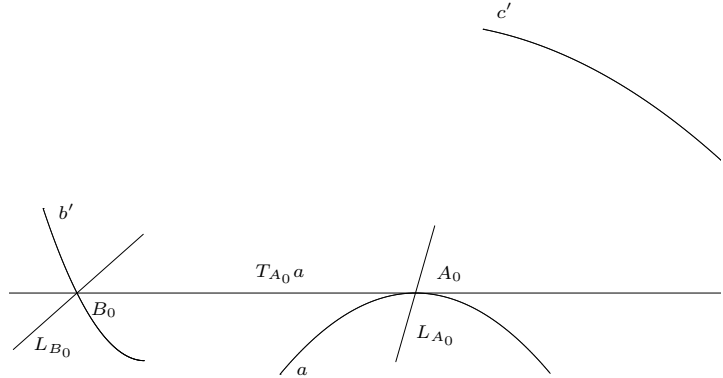}
\caption{The local projective billiard in the second case of Proposition \ref{prop:first_step_flat_orbit}: $T_{A_0}a$ intersects $b'$ transversally at $B_0$.}
\label{fig7}
\end{figure}

\begin{proof}
Consider the open subset of $M$ defined by
$$M^o = \{z\in M_{\text{reg}}\cap M_{\alpha}^{0}\sep d\pi_{G}(z) \text{ has rank } 1 \text{ on } D_{M}(z), G=\alpha,\beta,\gamma,B,C \}$$
where $\pi_{\beta}$, $\pi_{\gamma}:\alpha\times\mathcal{P}^2\to\mathcal{P}$ are the projections onto respectively the second and the third factor (in $\mathcal{P}$). By definition, $M^o$ contains an open subset of the integral surface $S$. In what follows by $S$ we denote the latter open subset and thus, consider that $M^o$ contains $S$. For $m_A\in\pi_\alpha(S)$, the set $W_{m_A}^o := M^o\cap W_{m_A}$ is an open subset of $W_{m_A}$ such that $W_{m_A}\setminus W_{m_A}^o$ is an analytic subset of $W_{m_A}$, hence is algebraic by Chow's theorem. Thus $W_{m_A}^o$ is a non-empty Zariski open subset of $W_{m_A}$. By Chevalley's theorem, $\pi_B(W_{m_A}^o)$ is a constructible subset, hence is a non-empty Zariski-open subset of $\pi_B(W_{m_A})$, which is itself either $\cp^2$ or an algebraic curve.

\begin{lemma}
\label{lemma:intersection_nondegenerate}
For $m_A\in\alpha$, we can choose $m_A'=(A',L_{A'})\in\alpha$ arbitrarily close to $m_A$, such that $T_{A'}a\cap\pi_B(W_{m_A}^o)$ is nonempty.
\end{lemma}

\begin{proof}
Fix $m_A\in\alpha$. By previous discussion, $Z_{m_A} := \pi_B(W_{m_A})\setminus\pi_B(W_{m_A}^o)$ is a strict algebraic subset of $\cp^2$. 

If $Z_{m_A}$ has dimension $0$, it is finite. Hence, since $a$ is not a line, the set $\Sigma$ of $m_A'\in\alpha$ such that $T_{A'}a$ does not intersect $Z_{m_A}$ is uncountable by Proposition \ref{prop:a_not_line}. By Bezout's theorem, for any $m_{A'}$ in $\Sigma$, $T_{A'}a$ intersects the algebraic subset $\pi_B(W_{m_A})$ of $\cp^2$ and the result is proved in this case.

If $Z_{m_A}$ has dimension $1$, $\pi_B(W_{m_A})$ is of dimension $2$, hence it equals $\cp^2$. Since $a$ is not a line, for any neighborhood $V$ of $m_A$ in $\alpha$, $Z_{m_A}$ cannot contain all of the $T_{A'}a$ for $m_{A'}\in V$ with $A'$ a regular point of $a$. Hence we can choose a $m_{A'}\in \alpha$ arbitrarily close to $m_A$ such that $T_{A'}a \nsubseteq Z_{m_A}$. Thus $T_{A'}a\cap\pi_B(W_{m_A}^o)$ is nonempty.
\end{proof} 

Fix an $m_A=(A,L_A)\in\alpha$. By Lemma \ref{lemma:intersection_nondegenerate} we can choose $m_{A_0}=(A_0,L_{A_0})\in\alpha$ arbitrarily close to $m_A$, such that $T_{A_0}a\cap\pi_B(W_{m_A}^o)$ is non-empty. This means that there is a $z\in W_{m_A}^o$ such that $B_0:=\pi_B(z)\in T_{A_0}a$. 

By Corollary \ref{cor:birkhoff_integrable}, there is a $2$-dimensional integrable surface $S_z$ through $z$. Now we can apply the same construction as in the proof of Proposition \ref{prop:link_birkhoff_billiard}: shrinking $S_z$ if necessary, $\hat{b}':=\pi_B(S_z)$ and $\hat{c}':=\pi_C(S_z)$ are immersed curves of $\cp^2$, $\pi_{\beta}(S_z)$ and $\pi_{\gamma}(S_z)$ are immersed curves of $\ptcp$, on which the restrictions $\pi_B:\pi_{\beta}(S_z)\to \hat{b}'$ and $\pi_C:\pi_{\gamma}(S_z)\to \hat{c}'$ are biholomorphisms, $L_B\neq T_B$, $L_C\neq T_C$. Therefore, the inverse maps of these restrictions, denoted by $\beta'$ and $\gamma'$ respectively, are line-framed curves, and $(\alpha,\beta',\gamma')$ is a $3$-reflective complex-analytic local projective billiard (Proposition \ref{prop:link_birkhoff_billiard}). The classical borders of $\beta$, $b':=\pi\circ\beta'$ is a curve in $\cp^2$ intersecting $T_{A_0}a$ (maybe tangentially) at $B_0$.

In the case when $A_0=B_0$ but the germs $(a,A_0)$ and $(b',B_0)$ do not coincide, we can choose $A_0'\in a\setminus b'$ sufficiently close to $A_0$ such that $T_{A_0'}a$ intersect $b'$ transversally at a point $B_0'$ close to $B_0$. By construction $A_0'\neq B_0'$. 
In any case, we can also move $A_0'$ a little so that $A_0'\notin F$ and have the same situation. Rename $A_0'$ and $B_0'$ by $A_0$ and $B_0$ to get the result.
\end{proof}

We will investigate from now on what happens when $m_B\in\beta'$ is close to $m_{B_0}$. We first can say that, under further conditions, the vertices will converge to a line. Thus we need the

\begin{definition}
Let $\mathcal{B}_0 = (\alpha_0,\beta_0,\gamma_0)$ be any local projective billiard. We say that $\mathcal{B}_0$ \textit{owns a flat orbit} if there are $m_{A}\in\alpha_0$, $m_B=(B,L_B)\in\beta_0$, $m_C=(C,L_C)\in\gamma_0$ and a line $L$ such that $A,B,C$ lie on $L$ and there exists a sequence of usual triangular orbits $(m_A,m_B^n,m_C^n)$ of $\mathcal{B}_0$ converging to $(m_A,m_B,m_C)$, such that for all $n$ $(m_A,m_B^n,m_C^n)$ belongs to an open set of triangular orbits of $\mathcal{B}_0$. The triple $(m_{A},m_B,m_C)$ is called a \textit{flat orbit} of $\mathcal{B}_0$ on $L$.
\end{definition}

\begin{proposition}
\label{prop:existence_flat_orbit}
Let $\mathcal{B} = (\alpha,\beta,\gamma)$ be a $3$-reflective complex-analytic local projective billiard such that $a$ is not supported by a line. Then there is a $3$-reflective complex-analytic local projective billiard $\mathcal{B}_0=(\alpha,\beta_0,\gamma_0)$ such that $\mathcal{B}_0$ owns a flat orbit. The corresponding flat orbit $(m_{A_0},m_{B_0},m_{C_0})$ has the following properties (see Figure \ref{fig8}):

1) The points $A_0=\pi(m_{A_0})$, $B_0=\pi(m_{B_0})$, $C_0=\pi(m_{C_0})$ lies on $T_{A_0}a$.

2) If two points among $\{A_0,B_0,C_0\}$ coincide, then the corresponding classical borders coincide. 

3) $T_{A_0}a$ intersect $b_0$ or $c_0$ transversally if $A_0\neq B_0$ or $A_0\neq C_0$ respectively. 
\end{proposition}

\begin{figure}[h]
\centering
\input{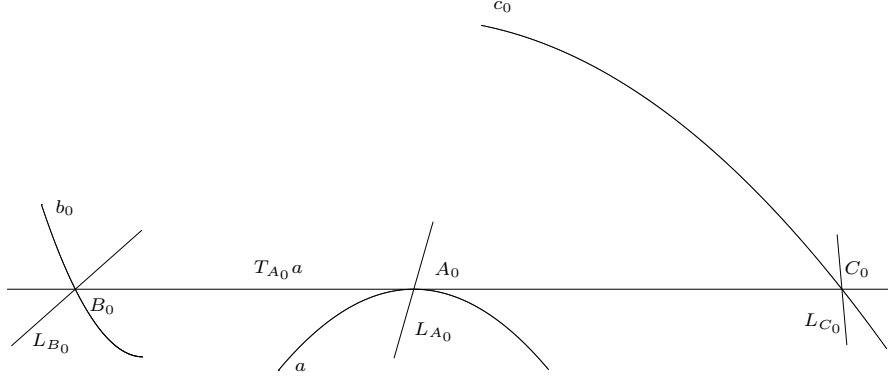}
\caption{The local projective billiard of Proposition \ref{prop:existence_flat_orbit} with a flat orbit $(m_{A_0},m_{B_0},m_{C_0})$ on $T_{A_0}a$. Here the three points $A_0$, $B_0$, $C_0$ are pairwise distinct.}
\label{fig8}
\end{figure}

\begin{proof}
Let $\mathcal{B}'=(\alpha,\beta',\gamma')$, $m_{A_0}=(A_0,L_{A_0})$, $m_{B_0}$ be respectively the local projective billiard and points from Proposition \ref{prop:first_step_flat_orbit}, verifying the further following conditions (which can be realized by choosing $F$ in a convenient way): 1) if $b'$ is supported by a line $\ell$, then $A_0\notin\ell$; 2) the set of points $m_B=(B,L_B)\in\beta'$ for which the line $L_B$ passes through $A_0$ is at most countable.

First let us construct $\gamma_0$ in the following way: for any $m_B=(B,L_B)\in\beta'$, denote by $A_0B^{\ast m_{A_0}}$ the line reflected from $A_0B$ at $m_{A_0}$ with respect to $(L_{A_0},T_{A_0}a)$, and by $A_0B^{\ast m_B}$ the line reflected from $A_0B$ at $m_B$ with respect to $(L_B,T_Bb')$. Denote by $K$ the (discrete) subset of $\beta'$ corresponding to the points where the lines $A_0B^{\ast m_A}$ and $A_0B^{\ast m_B}$ are the same or are not well-defined. Consider the map $\hat{\gamma}:\beta'\setminus K\to\cp^2$, which associates to any $m_B=(B,L_B)\in\beta'$ the point $C$ of intersection of the lines $A_0B^{\ast m_A}$ and $A_0B^{\ast m_B}$: $\hat{\gamma}$ is defined and analytic on $\beta'\setminus K$. Note that $\hat{\gamma}$ can be extended analytically to $m_{B_0}$: by solving the corresponding equations of lines defining $C$, we find that, around $m_{B_0}$, the coordinates of $\hat{\gamma}$ are rational fractions in the coordinates of $A_0B^{\ast m_A}$ and $A_0B^{\ast m_B}$.

Now $\hat{\gamma}$ is non-constant by the initial choice of $A_0$ (see conditions 1) and 2) imposed to $A_0$): the line $A_0B$ is not constant, so are $A_0B^{\ast m_A}$ and $A_0B^{\ast m_B}$. If $\hat{\gamma}$ were constant, it would be equal to $A_0$. But this would imply for all $B\in B$ that $A_0B^{\ast m_B}=A_0B$, hence that $A_0B=L_B$ or $T_Bb'$, contradicting the initial choice of $A_0$. Hence locally around each $m_B\notin K$, $\hat{\gamma}$ parametrises an analytic curve inside an open subset of $\cp^2$ and we can define $T(m_B)$ to be its tangent line at $\hat{\gamma}(m_B)$. Hence we can contruct an analytic map $\Gamma:\beta_0\setminus K\to\bref$ by setting $\Gamma(m_B)=(\hat{\gamma}(m_B),L_C(m_B))$ where $L_C(m_B)$ is the unique line through $C$ such that the lines $A_0B^{\ast m_A}$, $A_0B^{\ast m_B}$, $T(m_B)$ and $L_C(m_B)$ are harmonic (see Equation \eqref{eq:form_symmetry}). We set $\gamma_0=\Gamma$ and $\beta_0=\beta'$ which ends this first step.

Let us show that $(\alpha,\beta_0,\gamma_0)$ is a $3$-reflective local projective billiard. Indeed, $\mathcal{B}'=(\alpha,\beta',\gamma')$ is a $3$-reflective billiard, hence there are open subsets $U\subset\alpha$ and $V\subset\beta'$ verifying the following condition: for any $(m_A,m_B)\in U\times V$ there is an $m_C\in \gamma'$ such that $(m_A,m_B,m_C)$ is a $3$-periodic projective billiard orbit. As before, we can construct an analytic map $\hat{\gamma}':\alpha\times\beta_0\setminus K'\to\cp^2$, where $K'$ is a proper analytic subset of $\alpha\times\beta_0$, by setting $\hat{\gamma}'(m_A,m_B)$ to be the intersection point of $AB^{\ast m_A}$ and $AB^{\ast m_B}$ ($K'$ corresponds to the set of $(m_A,m_B)$ for which the lines $AB^{\ast m_A}$ and $AB^{\ast m_B}$ are the same or are not well-defined; note that $(m_{A_0},m_{B_0})\in K'$). 

Now $\hat{\gamma}'$ is of rank one on $U\times V$, hence on $\alpha\times\beta_0\setminus K'$, by connectedness of $\alpha\times\beta_0\setminus K'$ and by analyticity of the condition "being of rank at most one". As before, we can extend $\hat{\gamma}'$ into an analytic map $\Gamma':\alpha\times\beta_0\setminus K'\to\bref$, where $\Gamma'(m_A,m_B)=(\hat{\gamma}(m_A,m_B),L_C(m_A,m_B))$ and $L_C(m_A,m_B)$ is defined by the same rule as previously: it is the unique line through $C$ such that the lines $AB^{\ast m_A}$, $AB^{\ast m_B}$, $T(m_A,m_B)$ and $L_C(m_A,m_B)$ are harmonic, with $T(m_A,m_B)$ being the tangent line at $(m_A,m_B)$ of the germ of curve locally parmetrized by $\hat{\gamma}'$.
Again $\Gamma'$ is of rank one on $U\times V$, hence on $\alpha\times\beta_0\setminus K'$. In particular, for any $m_B\in\beta_0$ for which $(m_{A_0},m_B)\notin K'$, there is an open subset $U_0\times V_0\subset \alpha\times\beta_0$ containing $(m_{A_0},m_{B})$ such that $\Gamma'(U_0\times V_0)\subset\im \gamma_0$. Hence  $(\alpha,\beta_0,\gamma_0)$ is a $3$-reflective local projective billiard.

Finally, and by construction, when $m_B\to m_{B_0}$, $m_C:=\Gamma(m_B)$ is such that $(m_{A_0},m_B,m_C)$ is a triangular billiard orbit on $(\alpha,\beta_0,\gamma_0)$ accumulating on $T_{A_0}a$ (since $A_0B$ goes to $T_{A_0}a$ by definition, and thus also $A_0C=A_0B^{\ast m_A}$), and belonging to the interior of the set of triangular orbits by previous paragraph. Write $m_{C_0}=\lim_{m_B\to m_{B_0}}\Gamma(m_B)$, $b_0=\pi(\beta_0)$ and $c_0=\pi(\gamma_0)$. If $A_0=C_0$ and the germs $(c_0,C_0)$ and $(a, A_0)$ do not coincide, one can move $m_{A_0}$ a little so that $A_0\notin c_0$ and get the same conclusions with $A_0\neq C_0$. The same operation can be used to suppose that if $B_0=C_0$, the germs $(c_0,C_0)$ and $(b_0, B_0)$ coincide.
\end{proof}

\subsection{The $3$-reflective local projective billiard of Proposition \ref{prop:existence_flat_orbit} cannot exist}
\label{subsec:integral_not_reflective}

Let $\mathcal{B} = (\alpha,\beta,\gamma)$ be a $3$-reflective complex-analytic projective billiard such that $a$ is not supported by a line. Let $\mathcal{B}_0=(\alpha,\beta_0,\gamma_0)$, $m_{A_0}$, $m_{B_0}$, $m_{C_0}$ be respectively the $3$-reflective local projective billiard and points from Proposition \ref{prop:existence_flat_orbit}. Let $a,b_0,c_0$ be the classical borders of $\alpha,\beta_0,\gamma_0$ respectively. We want to show that $\mathcal{B}_0$ cannot exist.
 

To prove that, let us choose an affine chart $\mathbb C^2\subset\mathbb{CP}^2$ containing the points  $A_0$, $B_0$, $C_0$. Let us choose affine coordinates there so that
$$\az(T_{A_0}a)=0\qquad\text{ and }\qquad\infty\notin\{\az(T_{B_0}b_0),\az(T_{C_0}c_0),\az(L_{A_0})\}.$$
We will write until the end of this section
$$z=\az(A_0B), \quad z\etoile=\az(BC),\quad z'=\az(A_0C)$$
for any orbit $(m_{A_0},m_B,m_C)$ on $\mathcal{B}_0$ (see Figure \ref{fig9}, and section \ref{sec:reflection_law} for further details on azimuths). To be more precise, we will show the:

\begin{proposition}
\label{prop:three_equivalences}
When $m_B\to m_{B_0}$, the following asymptotic equivalences are satisfied: 
$$z' \sim (-z),\qquad z\etoile \sim (2I_b-1)z,\qquad z\etoile\sim (2I_c-1)z'$$
where $I_b$ (respectively, $I_c$) is the intersection index of $b$ (respectively $c$) with the tangent line $T_{A_0}a$ at $B_0$ (respectively $C_0$).
\end{proposition}

\noindent From Proposition \ref{prop:three_equivalences}, we deduce that $2I_c-1=-(2I_b-1)$ which is impossible since $2I_b-1$ and $2I_c-1$ are strictly positive integers. Hence $\mathcal{B}_0$ cannot exist. 

We will prove the three equivalences of Proposition \ref{prop:three_equivalences} in what follows, separated in three propositions (Propositions \ref{prop:equivalence_A}, \ref{prop:intersection_index} and \ref{prop:equivalence_azimuths}).

\begin{figure}[h]
\centering
\input{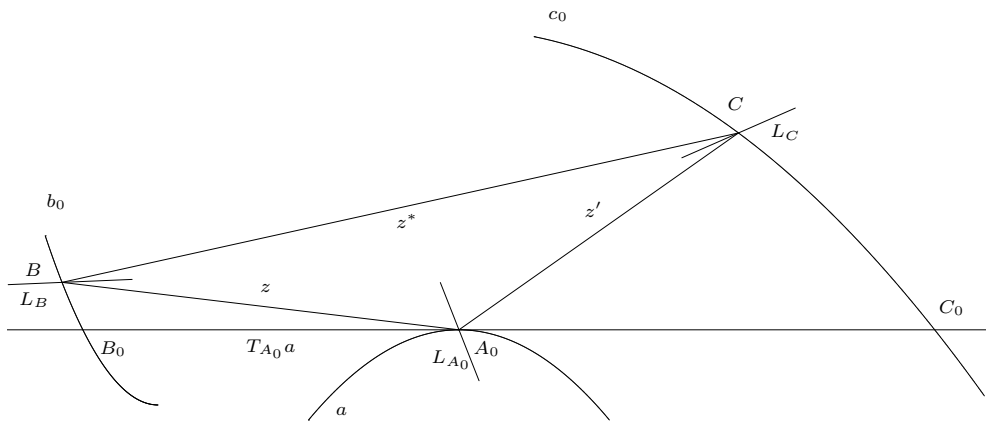}
\caption{The local projective billiard $\mathcal{B}_0$ with an orbit $(m_{A_0},m_B,m_C)$.}
\label{fig9}
\end{figure}

\begin{proposition}[$z'\sim (-z)$]
\label{prop:equivalence_A}
When $m_B=(B,L_B)\in\beta'$ goes to $m_{B_0}$, we have
$$z'\sim (-z).$$
\end{proposition}

\begin{proof}
Equation \eqref{eq:form_symmetry} of Proposition \ref{prop:symmetry} implies that 
$$z'=\frac{(\ell+t)z-2\ell t}{2z-(\ell+t)}$$
where $t = \az(T_{A_0}a)$, $\ell=\az(L_{A_0})$. By choice of coordinates, when $m_B\to m_{B_0}$,
$$z'=\frac{\ell z}{2z-\ell}\sim\frac{\ell z}{-\ell}=-z.$$
\end{proof}

\begin{proposition}
\label{prop:intersection_index}
Suppose $a$ is not supported by a line. If $A_0=B_0$ then when $m_B=(B,L_B)\in\beta'$ goes to $m_{B_0}$, we have
$$z\etoile\to (2I-1)z$$
where $I\geq2$ is the index of intersection of $a$ with the tangent line $T_{A_0}a$ at $A_0$.
\end{proposition}

\begin{proof}
Suppose $A_0=B_0$ : in this case $(b_0,B_0)=(a,A_0)$ by Proposition \ref{prop:existence_flat_orbit}. Take an orbit of the form $(m_{A_0},m_B,m_C)$ with $m_B$ and $m_C$ close to $m_{B_0}$ and $m_{C_0}$. Write $t = \az(T_Bb_0)$, $\ell=\az(L_B)$. Equation \eqref{eq:form_symmetry} of Proposition \ref{prop:symmetry} implies that 
$$\frac{z\etoile}{z}=\frac{(\ell+t)z-2\ell t}{z(2z-(\ell+t))}.$$
Now, when $m_B\to m_{B_0}$, since $a$ and $b_0$ are the same in a neighborhood of $B_0$, we can compute that $t\sim Iz$. Thus
$$\frac{z\etoile}{z}\sim \frac{(1-2I)\ell z}{-\ell z}=2I-1.$$
\end{proof}

\begin{lemma}
\label{lemma:distinct_points1}
Suppose $a$ is not supported by a line. If $B_0=C_0$, then the germs $(a,A_0)$, $(b_0,B_0)$, $(c_0,C_0)$ coincide.
\end{lemma}

\begin{proof}
Suppose that the three germs do not coincide and that $B_0=C_0$: in this case $(c_0,C_0)=(b_0,B_0)$ but $A_0\neq B_0$ with $T_{A_0}a$ intersecting $b_0$ transversally at $B_0$ by Proposition \ref{prop:existence_flat_orbit}. Take an orbit of the form $(m_{A_0},m_B,m_C)$ with $m_B$ and $m_C$ close to $m_{B_0}$ and $m_{C_0}$. Then, write $t = \az(T_Bb_0)$ and $\ell=\az(L_B)$. Equation \eqref{eq:form_symmetry} of Proposition \ref{prop:symmetry} implies that 
$$\ell=\frac{(z+z^{\ast})t-2zz^{\ast}}{2t-(z+z^{\ast})}$$
(indeed, $L_B$ and $T_Bb_0$ are symmetric with respect to $(A_0B,BC)$ by symmetry of the cross-ratio as noted in Remark \ref{remark:symmetric_cross_ratio}). Now, when $m_B\to m_{B_0}$, we have
$$z\to 0,\qquad t\to t_0$$
where $t_0:=\az(T_{B_0}b_0)\notin\{0,\infty\}$ by transversality of $b_0$ with $T_{A_0}a$ at $B_0$ and by choice of coordinates. But we also have
$$z^{\ast}\to t_0$$
because $BC\to T_{B_0}b_0$ since $B$,$C$ are distinct points (by definition of an orbit) of the same irreducible germ of curve $b_0=c_0$ converging to the same point $B_0=C_0$. Hence, when $m_B\to m_{B_0}$,
$$L\to \frac{t_0^2}{t_0}=t_0$$
which means that $L_{B_0}=T_{B_0}b_0$. But this is not the case by Proposition \ref{prop:existence_flat_orbit}, contradiction.
\end{proof}

\begin{proposition}
\label{prop:equivalence_azimuths}
Suppose $a$ is not supported by a line and $B_0\neq A_0$. Then when $m_B=(B,L_B)\in\beta$ goes to $m_{B_0}$, we have
$$z\etoile\sim z$$
which allows to extend the formula of Proposition \ref{prop:intersection_index} by setting $I=1$ in this case (transverse intersection).
\end{proposition}

\begin{proof}
First, let us prove the following lemma, which gives the form of the projective field of lines locally around $m_{B_0}$:

\begin{lemma}
\label{lemma:tangent_field_of_lines}
Suppose $a$ is not supported by a line and $B_0\neq A_0$. Then for $m_B=(B,L_B)\in\beta_0$ close to $m_{B_0}$, there is a $m_A=(A,L_A)\in\alpha$ close to $m_{A_0}$ for which $L_B$ is tangent to $a$ at $A$.
\end{lemma}

\begin{proof}
Proposition \ref{prop:existence_flat_orbit} implies that $T_{A_0}a$ intersects $b_0$ transversally at $B_0$. By the implicit function theorem, there is an analytic map which associates to any $m_{A}=(A,L_A)$ close to $m_{A_0}$ a point $m_B=(B,L_B)$ close to $m_{B_0}$ such that $B$ lies on $T_Aa$. Since $a$ is not a line, this map is not constant, hence is open and thus parametrizes (maybe non-bijectively) the germ of $\beta_0$ at $m_{B_0}$. Then we can choose $m_{A_1}$ in the neighborhood of $m_{A_0}$, and denote $m_{B_1}=(B_1,L_{B_1})$ the corresponding point on $\beta_0$ obtained via the above parametrization. 

We can suppose that $b_0$ is regular at $B_1$, that $T_{A_1}a$ is transverse to $b_0$ at $B_1$ and that $B_1\notin c_0$ (possible by Lemma \ref{lemma:distinct_points1}). Let $m_B=(B,L_B)\in\beta_0$ be a point converging to $m_{B_1}$. The line $A_1B$ converges to $T_{A_1}a$ and by the reflection law at $m_{A_1}$ we get that the line $A_1C$ also converges to $T_{A_1}a$, hence the limit $C_1$ of $C$ lies on $T_{A_1}a$. 
We also have that $C_1\neq B_1$ (Lemma \ref{lemma:distinct_points1}). 
Hence $C_1B_1=T_{A_1}a=A_1B_1$: $T_{A_1}a$ is invariant by the reflection law in $B_1$ with respect to $(L_{B_1},T_{B_1}b_0)$. Since it is transverse to $b_0$, we have $T_{A_1}a=L_{B_1}$, and this concludes the proof for $m_B=(B,L_B)\in\beta_0$ close to $m_{B_0}$.
\end{proof}

As in Lemma \ref{lemma:tangent_field_of_lines}, when $m_B$ is close to $m_{B_0}$, $L_B$ is tangent to $a$ at a point $A$ corresponding to a $m_A$ close to $m_{A_0}$: when $m_B$ converges to $m_{B_0}$, $m_A$ converges to $m_{A_0}$. Write $t = \az(T_Bb_0)$, $\ell=\az(L_B)$. We have by Formula \eqref{eq:form_symmetry} of Proposition \ref{prop:symmetry}, 
\begin{equation}
\label{eq:limit_symetry}
z\etoile = \frac{(\ell+t)z-2\ell t}{2z-(\ell+t)}.
\end{equation}
Now in this configuration, we easily compute using Lemma \ref{lemma:tangent_field_of_lines} that, when $m_B\to m_{B_0}$, 
$$\ell\sim z.$$ 
Here besides Lemma \ref{lemma:distinct_points1}, we essentially use the inequality $B_0\neq A_0$. This allows to use the following argument. As $B$ tends to $B_0$, the lines $L_B$ and $L_{B_0}=T_{A_0}a$ intersect at a point converging to $A_0$, while $B$ remains distant from $A_0$. This implies the required asymptotic equivalence of azimuths.

But we have also $t\to t_0$ where $t_0=\az(T_{B_0}b_0)\notin\{0,\infty\}$ (by the transversality condition of the intersection with $T_{A_0}a$). Hence, Equation \eqref{eq:limit_symetry} implies, when $m_B\to m_{B_0}$, that
$$\frac{z\etoile}{z} = \frac{(\ell+t)z-2t\ell}{z(2z-(\ell+t))}\sim\frac{-t_0z}{-t_0z}=1.$$
\end{proof}

\begin{proof}[Proof of Proposition \ref{prop:three_equivalences}]
The first asymptotic equivalence, $z\sim-z'$, comes from Proposition \ref{prop:equivalence_A}. The second one comes from Proposition \ref{prop:intersection_index}, when $I\geq 2$, and from Proposition \ref{prop:equivalence_azimuths}, when $I=1$. Finally the third one can be deduced from the second one by interchanging $m_B$ and $m_C$.
\end{proof}

\subsection{Conclusion}

Let $\mathcal{B} = (\alpha,\beta,\gamma)$ be a $3$-reflective complex-analytic local projective billiard, whose classical borders  are $a$, $b$, $c$. Supppose $a$ is not supported by a line. Let $\mathcal{B}_0=(\alpha,\beta_0,\gamma_0)$ be the $3$-reflective local projective billiard of Proposition \ref{prop:existence_flat_orbit}, and $a,b_0,c_0$ be the classical borders of $\alpha,\beta_0,\gamma_0$ respectively. 

Now by Proposition \ref{prop:three_equivalences}, we deduce that $2I_c-1=-(2I_b-1)$ which is impossible since $2I_b-1$ and $2I_c-1$ are strictly positive integers. Hence $\mathcal{B}_0$ cannot exist, contradiction: $a$ is supported by a line. By symmetry of previous argument, $a$, $b$, an $c$ are supported by lines, which proves Proposition \ref{prop:classical_borders_lines}. 

Finally, by Proposition \ref{prop:reflectivity_and_lines1}, $\mathcal{B}$ is a right-spherical billiard as defined in Definition \ref{def:spherical_billiard}. Hence Theorem \ref{thm:main_theorem_analytic} is proved.

\section{Analytic case: proof of Theorem \ref{thm:main_theorem_multidim_analytic}}
\label{sec:multidim}

In this section, we want to prove Theorem \ref{thm:main_theorem_multidim_analytic} using Theorem \ref{thm:main_theorem_analytic}. Fix $d\geq 3$ to be the dimension in which the problem takes place. We will need this auxiliary lemma:

\begin{lemma}
\label{lemma:lines_and_tangent}
Let $W\subset\cmplx^d$ be an analytic hypersurface, $p\in W$ and $U$ a non-empty open subset of $\proj(T_pW)$. Suppose that for any $\proj(v)\in U$, $W$ contains the points $p+tv$ for all $t$ sufficiently close to $0$ (dependenlty on $v$). Then $W$ is an hyperplane.
\end{lemma}

\begin{proof}
We can suppose that $p=0$, $T_pW={z_d=0}$ and $W$ is locally the graph of an analytic map $f:V\to\cmplx$ where $V\subset\cmplx^{d-1}$ is an open subset containing $0$. Let $v\in\cmplx^{d-1}$ be a non-zero vector such that $\proj(v)\in U$. By assumption, for $t$ close to $0$ we have $g_v(t):=f(tv)=0$. Since $g_v$ is analytic, it is $0$ everywhere where it is defined. Yet the set $\{tv\sep t\in\real, \proj(v)\in U\}$ contains a non-empty open subset of $V$, on which $f$ should vanish. By analyticity $f=0$ and $W$ is the hyperplane defined by the equation $z_d=0$.
\end{proof}

\begin{proof}[Proof of Theorem \ref{thm:main_theorem_multidim_analytic}]
Suppose by contradiction that there is a $3$-reflective complex-analytic local projective billiard $\mathcal{B}=(\alpha,\beta,\gamma)$ in $\proj(T\cmplx^d)$. Let $a=\pi(\alpha)$, $b=\pi(\beta)$ and $c=\pi(\gamma)$. Let $U\times V\subset \alpha\times\beta$ be an open subset for which all $(m_{A},m_{B})\in U\times V$ can be completed in $3$-periodic orbits of $\mathcal{B}$. Let us state the following lemma, which isn't worth of a proof:

\begin{lemma}
\label{lemma:planar_orbits}
Let $(m_A,m_B,m_C)$ be a $3$-periodic orbit of $\mathcal{B}$ where $m_A=(A,L_A)$, $m_B=(B,L_B)$, $m_C=(C,L_C)$. Then all lines $AB$, $BC$, $CA$, $L_A$, $L_B$, $L_C$ belong to the plane $ABC$, which is transverse to $a,b,c$ at $A,B,C$ respectively.
\end{lemma}

First let us show the

\begin{lemma}
\label{lemma:hypersurf_planes}
The hypersurfaces $a$ and $b$ are supported by hyperplanes.
\end{lemma}

\begin{proof}
By symmetry, let us just show that $a$ is supported by a hyperplane. Fix $m_A\in U$. For $m_B\in V$, consider the plane $P_{m_B}$ containing the triangular orbit starting by $(m_A,m_B)$, as in Lemma \ref{lemma:planar_orbits}. Consider $a_{m_B}$, $b_{m_B}$, $c_{m_B}$ to be the intersections of $P_{m_B}$ respectively with $a$, $b$, $c$: by transversality, and shrinking them if needed, we can suppose that they are immersed curves of $P_{m_B}$.

Now let $\alpha_{m_B}=\pi\ante(a_{m_B})\cap\alpha$. Since $\alpha$ and $a$ have the same dimension, $\pi_{|\alpha}:\alpha\to a$ is a diffeomorphism, and thus $\alpha_{m_B}$ is an immersed curve.
Define $\beta_{m_B}$ and $\gamma_{m_B}$ similarly. Let us show that $(\alpha_{m_B}, \beta_{m_B}, \gamma_{m_B})$ is a planar $3$-reflective projective billiard. Consider the open subsets $U'=U\cap\alpha_{m_B}$ of $\alpha_{m_B}$ and $V'=V\cap\beta_{m_B}$ of $\beta_{m_B}$. Any $m_B'\in V'$ is such that $(m_A,m_B')$ can be completed in an orbit $(m_A,m_B',m_C')$ of $\mathcal{B}$ and by Lemma \ref{lemma:planar_orbits}, $AB'C'$ defines a plane containing $L_A$ and $AB'$, which are intersecting lines inside $P_{m_B}$. Hence $AB'C'=P_{m_B}$ and thus $\beta_{m_B}$ is an analytic curve such that for all $m_B'=(B',L_B')\in V'$, $B'$ and $L_B'$ are in $P_{m_B}$. The same argument work for $\alpha_{m_B}$; hence also for $\gamma_{m_B}$ by completing any $(m_A',m_B')\in U'\times V'$ into a $3$-periodic orbit. Finally we showed that $(\alpha_{m_B}, \beta_{m_B}, \gamma_{m_B})$ is a $3$-reflective projective billiard inside $P_{m_B}$.

In particular, by Theorem \ref{thm:main_theorem_analytic}, $a_{m_B}$ is supported by a line, $\ell_{m_B}$. Now $\ell_{m_B}$ is included in $T_{A}a$ (since the tangent space of $a_{m_B}$ is included in the tangent space of $a$) and in $P_{m_B}$. This result is true for any $m_B\in V$, implying the same result for lines in a neighborhood of $\ell_{m_B}$ in $T_Aa$: hence by Lemma \ref{lemma:lines_and_tangent}, $a$ is supported by an hyperplane, which concludes the proof.
\end{proof}

Let $\hat{a}$ be the hyperplane supporting $a$ and $\hat{b}$ be the hyperplane supporting $b$.

\begin{lemma}
There is a point $B_0\in\hat{b}$ such that for all $m_A=(A,L_A)\in\alpha$ the line $L_A$ goes through $B_0$.
Similarly, there is a point $A_0\in\hat{a}$ such that for all $m_B=(B,L_B)\in\beta$ the line $L_B$ goes through $A_0$.
\end{lemma}

\begin{proof}
Let us show the existence of $B_0$, the existence of $A_0$ being analogous. Fix $m_A=(A,L_A)\in U$, and consider the point $B_0\in\hat{b}$ of intersection of $L_A$ with $\hat{b}$. For $m_B\in V$, consider the plane $P_{m_B}$ containing the triangular orbit starting by $(m_A,m_B)$, as in Lemma \ref{lemma:planar_orbits}: define $a_{m_B}$, $b_{m_B}$, $c_{m_B}$, $\alpha_{m_B}$, $\beta_{m_B}$, $\gamma_{m_B}$, $U'$, $V'$  as in the proof of Lemma \ref{lemma:hypersurf_planes}.  One has $B_0\in b_{m_B}\subset P_{m_b}$, by Lemma \ref{lemma:planar_orbits}.

We recall that $(\alpha_{m_B}, \beta_{m_B}, \gamma_{m_B})$ is a planar $3$-reflective projective billiard. By Theorem \ref{thm:main_theorem_analytic} it is a right-spherical billiard, hence each $m_A'=(A',L_A')\in U'$ is such that $L_A'$ and $L_A$ intersect $b_{m_B}$ at the same point. By construction, this should be the point $B_0=L_A\cap b_{m_B}$. Therefore, any $m_A'=(A',L_A')\in U'$ is such that $L_A'$ passes through $B_0$. Hence by analyticity, if $\ell_{m_B}$ is the line of intersection of $P_{m_B}$ with $a$, every $m_A'=(A',L_A')\in \alpha\cap\pi\ante(\ell_{m_B})$ is such that $L_A'$ passes through $B_0$.

Now the union of all $\ell_{m_B}$ for $m_B\in V$ contains a non-empty open subset $\Omega$ of $a$, which by construction has the following property: all $m_A'=(A',L_A')\in \alpha\cap\pi\ante(W)$ is such that $L_A'$ passes through $B_0$. By analyticity, all $m_A'=(A',L_A')\in \alpha$ is such that $L_A'$ passes through $B_0$, and the proof is complete.
\end{proof}

Now we can finish the proof of Theorem \ref{thm:main_theorem_multidim_analytic}. Indeed, any $z=(m_A,m_B)\in U\times V$, can be completed in an orbit, which lies in a plane $P_z$. This plane $P_z$ contains $L_A$ and $L_B$ (Lemma \ref{lemma:planar_orbits}), hence goes through $A_0$ and $B_0$. 
If $A_0\neq B_0$, $P_z = AA_0B_0$, but this is impossible since in this case $P_z$ doesn't depend on $m_B$, which can be chosen such that $B\notin AA_0B_0$. Hence $A_0=B_0$ and $A_0\in\hat{a}$, implying that all $m_A=(A,L_A)\in U$ are such that $L_A\subset T_Aa$. This contradicts the definition of $\alpha$, and the theorem is proved.
\end{proof}

\section{Smooth case: proof of Theorems \ref{thm:main_theorem} and \ref{thm:main_theorem_multidim}}
\label{sec:smooth_case}

We want to prove Theorem \ref{thm:main_theorem} and Theorem \ref{thm:main_theorem_multidim} by extending Theorem \ref{thm:main_theorem_analytic} and Theorem \ref{thm:main_theorem_multidim_analytic} from the analytic to the smooth case.
%
%

\subsection{Lifts of submanifolds in the grassmanian bundle}

Let $M$ be a smooth manifold. Let $r,k\in\pinteger$. Consider the canonical projection $\pi: \gr_k(TM)\to M$ on the grassmanian bundle.

\begin{definition}
\label{def:first_lift}
Let $S\subset M$ be a $\class^{r}$-smooth injectively immersed submanifold of $M$ of dimension $k$. The \textit{first lift of $S$} is the injectively immersed $\class^{r-1}$-submanifold $S^{(1)}\subset\gr_k(TM)$, transverse to $\pi$, defined by
$$S^{(1)} = \{(x,T_xS)\sep x\in S\}.$$
\end{definition}

\begin{remark}
If $S$ is parametrized by a $\class^{r+1}$-smooth injective immersion $f:U\subset\real^k\to S$, then $S^{(1)}$ is parametrized by the $\class^r$-smooth injective immersion $f^{(1)}:U\subset\real^k\to S^{(1)}$ defined for all $x\in U$ by $f^{(1)}(x) = (f(x),\im df(x))$, and called \textit{first lift of $f$}.
\end{remark}

In what follows, denote by $J^r(X,Y)$ the space of $r$-jets of $\class^r$-smooth maps from a given manifold $X$ to another one $Y$, together with its usual topology.

\begin{lemma}
\label{lemma:jet_prol_convergence}
Let $f,g:U\subset\real^k\to\real^d$ be $\class^{r+1}$-smooth injective immersions, and let $0\in U$. Then for any neighborhood $V_0\subset J^{r+1}(\real^k,\real^d)$ containing the $(r+1)$-jet of $f$ at $0$, there is a neighborhood $V_1\subset J^r(\real^k,\gr_k(\real^d))$ depending only on $f$, containing the $r$-jet of $f^{(1)}$ at $0$, and verifying the following property: if the $r$-jet of $g^{(1)}$ at $0$ is in $V_1$, then there is a smooth diffeomorphism $\varphi:W\subset\real^k \to W'\subset U$, where $0\in W'=\varphi(W)$ for which the $(r+1)$-jet of $g\circ\varphi$ at $\varphi\ante(0)$ is in $V_0$.
\end{lemma}

\begin{remark}
Therefore in terms of jets, the pointwise $\class^r$-convergence of first lifts implies the pointwise $\class^{r+1}$-convergence of the inital maps modulo reparametrization by diffeomorphisms.
\end{remark}

\begin{proof}
It is sufficient to prove Lemma \ref{lemma:jet_prol_convergence} when we replace $f$ and $g$ by $\psi\circ f\circ\psi'$ and by $\psi\circ g\circ\psi'$ for any diffeomorphisms $\psi:\real^d\to\real^d$ and $\psi':W\subset\real^k \to W'\subset U$, where $0\in W'=\psi'(W)$. Therefore we can suppose that $f$ is of the form $f:x\in U\mapsto (x,f_0(x))$, where $f_0:U\to\real^{d-k}$ is $\class^{r+1}$-smooth.

Now since $\pi\circ g^{(1)}=g$, we can choose a first neighborhood $V_1$ of $f^{(1)}$ such that if the $r$-jet of $g^{(1)}$ at $0$ is in $V_1$, then one can find a $\class^{\infty}$-smooth diffeomorphism $\varphi:W\subset\real^k \to W'\subset U$, sending $0$ on $0$, such that $g\circ\varphi$ is of the form $x\in W\mapsto (x,g_0(x))$, where $g_0:W\to\real^{d-k}$ is $\class^{r+1}$-smooth. Notice that $\im d(g\circ\phi) = \im dg$, hence $\im dg$ is generated by the vectors $B=(e_i,\partial_i g_0)$, $i=1\ldots k$, where $e_1,\ldots,e_d$ is the canonical basis of $\real^d$ and $\partial_i$ is the $i$-th partial derivative. Similarly $\im df$ is generated by the vectors $(e_i,\partial_i f_0)$, $i=1\ldots k$. 

Now $B$ defines a set of coordinates in $\gr_k(\real^d)$ for which the coordinates of $\im df$ are the coordinates of $(\partial_1 f_0,\ldots,\partial_k f_0)$ in $(e_{d-k},\ldots,e_d)$, and the coordinates of $\im dg$ are the coordinates of $(\partial_1 g_0,\ldots,\partial_k g_0)$ in $(e_{d-k},\ldots,e_d)$. Therefore, saying that the $r$-jet of $\im dg$ at $0$ is close to the $r$-jet of $\im df$ at $0$ means that the same holds for the $r$-jets at $0$ of the partial derivatives $(\partial_1 g_0,\ldots,\partial_k g_0)$ and $(\partial_1 f_0,\ldots,\partial_k f_0)$. And with the additionnal assumption that $g(0)$ is close to $f(0)$, this means that the $(r+1)$-jets of $f$ and $g\circ\varphi$ at $0$ are close.
\end{proof}

\begin{definition}
Define $\mathcal{G}(M, k, r)$ to be the set of all germs of $k$-dimensional $\class^r$-smooth submanifolds of $M$. For any open subset $\Omega\subset J^r(\real^k, M)$ of $r$-jets of mappings $\real^k\to M$, we define the subset $\mathcal{G}(\Omega)\subset\mathcal{G}(M, k, r)$ to be the set of germs $(S,p)$ of $\class^r$-smooth submanifolds for which one can find an injective immersion $f:\real^k\to S$ with $f(0)=p$ and the $r$-jet of $f$ at $0$ being in $\Omega$.

The topology generated by all $\mathcal{G}(\Omega)$ will be called \textit{Whitney $\class^r$-topology on $\mathcal{G}(M, k, r)$}.
\end{definition}

\begin{remark}
The Whitney $\class^r$-topology on $\mathcal{G}(M, k, r)$ is not Hausdorff: if two germs $(S,p)$ and $(S',p)$ of $\class^r$-smooth submanifolds have the same $r$-jets, then any neighborhood of $(S,p)$ contains $(S',p)$.
\end{remark}

We deduce the following proposition from Lemma \ref{lemma:jet_prol_convergence}.

\begin{proposition}
Let $S,S'\subset M$ be two $k$-dimensional $\class^{r+1}$-smooth submanifolds and $(p,p')\in S\times S'$. Then for any open subset $V_0\subset\mathcal{G}(M, k, r+1)$ containing $\left(S, p\right)$, there is an open subset $V_1\subset\mathcal{G}(\gr_k(TM), k, r)$ depending only on $S$ and containing $\left(S^{(1)}, (p,T_pS)\right)$ such that 
$$\left(S'^{(1)}, (p',T_{p'}S')\right)\in V_1 \Rightarrow (S',p')\in V_0.$$
\end{proposition}

\subsection{Prolongation of Pfaffian systems and pointwise $\class^r$-approximation by analytic integral surfaces}

Suppose $M$ is an analytic manifold. Let $\mathcal{D}$ be an analytic distribution on $M$. Fix a $k\in\{1,\ldots,\dim \mathcal{D}\}$. Recall that an \textit{integral} $k$-dimensional surface of $\mathcal{D}$ is an $k$-dimensional immersed submanifold $S\subset M$ such that for all $x\in S$, $T_xS\subset\mathcal{D}(x)$. The latter can be $\class^r$, $\class^{\infty}$, or analytic.

\begin{definition}[\cite{glutkud2}, definition 21]
Given a family of analytic distributions $(\mathcal{D}_i)_i$ on $M$, we call the data $\mathcal{P} = (M,\mathcal{D},k;(\mathcal{D}_i)_i)$ a \textit{Pfaffian system with transversality conditions}. An integral surface of $\mathcal{P}$ is an \textit{integral surface of} $\mathcal{D}$ of dimension $k$ such that, for all $i$, its tangent subspaces either are transverse to $\mathcal{D}_i$, or intersect it by zero.
\end{definition}


\begin{definition}
The \textit{contact distribution} on $\gr_k(TM)$ is the distribution $\mathcal{K}$ on $\gr_k(TM)$ defined for all $(x,E)\in\gr_k(TM)$ by $\mathcal{K}(x,E) = d\pi\ante(E).$
\end{definition}

\begin{definition}[\cite{glutkud2}, definition 12]
A $k$-plane $E\in\gr_k(TM)$ is said to be \textit{integral} if for any $1$-form $\omega$ vanishing on $\mathcal{D}$, $d\omega$ vanishes on $E$.
\end{definition}

As described in \cite{glutkud2}, subsection 2.3, the set $\widetilde{M}_k$ of integral $k$-planes of $M$ is an analytic subset hence a stratified manifold: it is a locally finite and at most countable disjoint union of smooth analytically constructible subsets (see \cite{lojasiewicz}, section IV.8).

\begin{definition}[\cite{glutkud2}, definition 23]
Let $\mathcal{P} = (M,\mathcal{D},k;(\mathcal{D}_i)_i)$ be a Pfaffian system with transversality conditions. Let $M'$ be a stratum of $\widetilde{M}_k$, $\mathcal{K}'$ the restriction of $\mathcal{K}$ to $M'$, and $\mathcal{D}'_i$ the pull-back of  $\mathcal{D}_i$ on $M'$ for each $i$. The Pfaffian system $\mathcal{P}' = (M',\mathcal{K}',k;(\mathcal{D}'_i)_i,\ker d\pi)$ is called \textit{a first Cartan prolongation} of $\mathcal{P}$.
\end{definition}

First Cartan prolongations have the following property:

\begin{proposition}[\cite{glutkud2}, subsection 2.3, \cite{bryant_chern}, chapter VI]
\label{prop:lift_prol}
Let $\mathcal{P}=(M,\mathcal{D},k;(\mathcal{D}_i)_i)$ be a Pfaffian system with transversality conditions. The lift $S^{(1)}$ of an integral surface $S$ of $\mathcal{P}$ contains an open dense subset such that each its connected component $S'$ lies in some stratum $M'$ of $\widetilde{M}_k$, and such that $S'$ is an integral surface of the first Cartan prolongation $\mathcal{P}' = (M',\mathcal{K}',k;(\mathcal{D}'_i)_i,\ker d\pi)$. We call again $S'$ \textit{a first Cartan prolongation} of $S$.

\end{proposition}

\begin{corollary}
\label{cor:infinite_prol}
Let $\mathcal{P}$ be the same, as in Proposition \ref{prop:lift_prol}, and let $S\subset M$ be a $C^{\infty}$-smooth 
 integral surface of the Pfaffian system $\mathcal{P}$. Then there exist an infinite sequence of Pfaffian systems $\mathcal{P}^{(r)}$ 
 and an infinite sequence of non-empty integral surfaces $S_r$ of $\mathcal{P}^{(r)}$, $S_0=S$, 
  such that for every $r$ the system  $\mathcal{P}^{(r+1)}$ is 
 a first Cartan prolongation of $\mathcal{P}^{(r)}$, and the surface $S_{r+1}$  is a first Cartan prolongation of the 
 surface $S_r$.
\end{corollary}

\begin{proof}
We deduce Corollary \ref{cor:infinite_prol} from Proposition \ref{prop:lift_prol} by induction: if $\mathcal{P}^{(r)}$ and $S_r$ are defined, choose $S_{r+1}$ to be a Cartan prolongation of $S_r$ which lies in a Cartan prolongation $\mathcal{P}^{(r+1)}$ of $\mathcal{P}^{(r)}$.
\end{proof}

We quote the following theorem, related to the existence of analytic integral surfaces of Pfaffian systems, which is cited in \cite{glutkud2}, theorem 24, and in \cite{bryant_chern}, chapter VI, paragraph 3.

\begin{theorem}[E. Cartan \cite{cartan}, M. Kuranishi \cite{kuranishi}, P. K. Rashevsky \cite{rashevsky}]
\label{theorem:cartan_ku_ra}
Let $\mathcal{P}=(M,\mathcal{D},k;(\mathcal{D}_i)_i)$ be a Pfaffian system with transversality conditions. Suppose that $\mathcal{P}$ has no analytic integral surfaces. Then for any sequence of Pfaffian systems $\mathcal{P}^{(r)}=(M^{(r)},\ldots)$ such that $\mathcal{P}^{(r+1)}$ is a first Cartan prolongation of $\mathcal{P}^{(r)}$, one can find an $r_0$ for which $M^{(r_0)}=\emptyset$.
\end{theorem}


We deduce from this theorem the following

\begin{proposition}
\label{prop:analytic_approx}
Let $r\in\pinteger$, $S\subset M$ be a $\class^{\infty}$-smooth integral surface of the Pfaffian system $\mathcal{P}=(M,\mathcal{D},k;(\mathcal{D}_i)_i)$ and $p\in S$. Then for any open subset $V\subset\mathcal{G}(M, k, r)$ containing $(S,p)$ one can find an analytic integral surface $S'\subset M$ of $\mathcal{P}$ and $p'\in S'$ such that $(S',p')\in V$.
\end{proposition}

\begin{remark}
Proposition \ref{prop:analytic_approx} means that the $r$-jet of $S$ at $p$ can be approximated by the $r-$jets of germs of analytic integral surfaces of $\mathcal{D}$.
\end{remark}

\begin{proof}
It is enough to prove Proposition \ref{prop:analytic_approx} for $r=0$, since we can apply it to the $r$-th prolongation of $S$ and use Lemma \ref{lemma:jet_prol_convergence} to conclude the same on the $r$-jets of $S$. Thus we need to show that for any open subset $V\subset M$ containing $p$, one can find an analytic integral surface $S'$ of $\mathcal{P}$ intersecting $V$. 

But $V$ contains the $\class^{\infty}$ integral surface $S\cap V$ of $\mathcal{P}$. Therefore, one can find a sequence of prolongations $\mathcal{P}^{(r)}=(M^{(r)},\ldots)$ of the Pfaffian system $\mathcal{P}=(M,\mathcal{D},k;(\mathcal{D}_i)_i)$ such that $M^{(r)}\neq \emptyset$ for all $r$ (Corollary \ref{cor:infinite_prol}). Therefore Theorem \ref{theorem:cartan_ku_ra} implies the existence of an analytic integral surface of $\mathcal{P}$ in $V$.

%
%
\end{proof}

\subsection{Birkhoff's ditribution}

We define another version of Birkhoff's distribution for projective billiards, which will be useful. Let $d\geq 2$. Denote by $M$ the open subset of $\proj(T\rp^d)^3$ constituted by the elements $z=(A,L_A,B,L_B,C,L_C)$ where $A$, $B$, $C$ are not on the same line and $L_A\neq AB,AC$, $L_B\neq BA,BC$, $L_C\neq CA, CB$. Write $proj_G:M\to\proj(T\rp^d)$ and $\pi_G:M\to\rp^d$ the maps such that $proj_G(z)=(G,L_G)$ and $\pi_G(z)=G$ for $G=A,B,C$. Define Birkhoff's distribution $\mathcal{D}$ on $M$ by 
$$\mathcal{D}(z) = d\pi_A\ante(T_A)\cap d\pi_B\ante(T_B)\cap d\pi_C\ante(T_C)$$ where $T_A(z)$ is the only line through $A$ such that the quadruple of lines $(AB,AC,L_A,T_A)$ is harmonic, and $T_B(z), T_C(z)$ are defined analogously. 

Now, consider the subset $M\etoile\subset\gr_{2(d-1)}(M)$ made by the $2(d-1)$-dimensional integral planes $E\subset TM$ of $\mathcal{D}$, on which $dproj_G$ has rank $d-1$ for $G=A,B,C$: it is an open dense subset of an analytic subset. One can define the contact distribution $\mathcal{K}$ on $M\etoile$ by
$$\mathcal{K}(z,E) = d\pi\ante(E)$$ 
where $\pi:\gr_{2(d-1)}(M)\to M$ is the canonical projection. The distribution $\mathcal{K}$ has the following property related to $3$-reflective projective billiards:

\begin{proposition}[Version of \cite{glutkud2} Lemma 18 for projective billiards]
\label{prop:distribution_pseudo_billard}
Consider a $3$-reflective projective billiard $(\alpha,\beta,\gamma)$. Let $S$ be a subset of triangular orbits parametrized by an open subset of $\alpha\times\beta$. Then $S\subset M$ is a $2(d-1)$-dimensional injectively immersed submanifold, its first lift $S^{(1)}$ (as in Definition \ref{def:first_lift}) lies in $M\etoile$ and is an integral surface of the Pfaffian system $(M\etoile,\mathcal{K},2(d-1);\ker d\pi)$.
\end{proposition}

\begin{proof}
Consider the projective billiard map $B:\alpha\times\beta\to\gamma$. By definition, one can find an open subset $U\subset\alpha\times\beta$ such that $S=\{(p,q,B(p,q))\sep (p,q)\in U\}$. Hence $S$ is parametrized by the injective immersion $s:U\to M$ defined by $s(p,q)= (p,q,B(p,q))$, which makes $S$ a $\class^{\infty}$-smooth $2(d-1)$-dimensional injectively immersed submanifold. For $G=A,B,C$, $dproj_G$ has rank $d-1$ on $TS$ since it sends $S$ to $\alpha$, $\beta$ or $\gamma$. 

Now $S$ is an integral surface of $\mathcal{D}$ since we can easily state the following equivalence: for $z\in S$, $T_zS\subset\mathcal{D}(z)$ if and only if $z$ is a $3$-periodic orbit of $(\alpha,\beta,\gamma)$. Therefore the tangent planes of $S$ are integral and $S^{(1)}\subset M\etoile$. We now easily have that $S^{(1)}$ is also a $\class^{\infty}$-smooth integral surface of $\mathcal{K}$ transverse to $\pi$.
\end{proof}

\begin{proposition}[Version of \cite{glutkud2} Lemma 17 for projective billiards]
\label{prop:reflective_billiard_distribution}
Let $S\etoile$ be an integral surface of dimension $2(d-1)$ of the Pfaffian system $(M\etoile,\mathcal{K},2(d-1);\ker d\pi)$. Then any point in $S\etoile$ is contained in an open subset $U$ of $S\etoile$ for which $(proj_A\circ\pi(U),proj_B\circ\pi(U),proj_C\circ\pi(U))$ is a local $3$-reflective projective billiard.
\end{proposition}

\begin{proof}
Since $S\etoile$ is transverse to $\pi$ and $proj_G$ has rank $d-1$ on the planes constituting $M\etoile$, any $U$ sufficiently small is such that $\pi(U)$ is an integral surface of $\mathcal{D}$ and $proj_G(\pi(U))$ is an immersed submanifold of $\proj(T\rp^d)$ of dimension $d-1$ for $G=A,B,C$. Now the same argument as in the proof of Proposition \ref{prop:distribution_pseudo_billard} leads to the result that $(proj_A\circ\pi(U),proj_B\circ\pi(U),proj_C\circ\pi(U))$ is a local $3$-reflective projective billiard.
\end{proof}

\begin{remark}
\label{remark:construction_pseudo}
In the case when $d=2$, both Proposition \ref{prop:distribution_pseudo_billard} and Proposition \ref{prop:reflective_billiard_distribution}, together with the existence of a 3-reflective billiard, imply the existence of local projective billiards having a set of triangular orbits of non-zero measure (in the space of initial conditions), and which are not right-spherical billiards. Consider the first lift $S^{(1)}_{RS}\subset M\etoile$ of the surface $S_{RS}$ corresponding to triangular orbits of a right spherical billiard, as in Proposition \ref{prop:distribution_pseudo_billard}. Choose a surface $S\etoile\subset M\etoile$ which coincide with $S^{(1)}_{RS}$ only on a compact set which has empty interior but is of non-zero measure (for example a specific Cantor set). Then $S\etoile$ is what is called a \textit{pseudo-integral surface} of $\mathcal{K}$: the set $V$ of points $p\in S^{(1)}_{RS}$ for which $T_pS^{(1)}_{RS}\subset\mathcal{K}(p)$ has non-zero measure (in $S^{(1)}_{RS}$). Now, as in the proof of Proposition \ref{prop:reflective_billiard_distribution}, 
to any $p\in V$ corresponds a neighborhood $p\in U\in S\etoile$ projecting to line-framed hypersurfaces $\alpha,\beta,\gamma$, and the latter have the property announced in Remark \ref{remark:introducing_contruction_pseudo}: the corresponding local projective billiard has a set of triangular orbits of non-zero measure (in the space of initial conditions), obtained by projecting $V$ to $M$.
\end{remark}

\begin{proof}[Proof of Theorem \ref{thm:main_theorem}]
Here $d=2$. Let $\mathcal{B}=(\alpha, \beta,\gamma)$ be a local $\class^{\infty}$-smooth $3$-reflective planar projective billiard in the plane, and suppose that it is not right-spherical. Let $z$ be the orbit formed by its basepoints. Define $S^{(1)}$ to be the first lift of the subset $S$ from Proposition \ref{prop:distribution_pseudo_billard}, which is a $\class^{\infty}$-smooth integral surface of the Pfaffian system $\mathcal{P}=(M\etoile,\mathcal{K},2,\ker d\pi)$.

Let $r\in\pinteger$ and $p\in S^{(1)}$. By Proposition \ref{prop:analytic_approx}, for any open subset $V\subset\mathcal{G}(M\etoile, 2, r)$ containing $(S^{(1)},p)$ one can find an analytic integral surface $S\etoile_a\subset M$ of $\mathcal{P}=(M\etoile,\mathcal{K},2;\ker d\pi)$ and $p'\in S\etoile_a$ such that $(S\etoile_a,p')\in V$. By Proposition \ref{prop:reflective_billiard_distribution}, any small neighborhood $U=U(p')$ of the point $p'$ in $S\etoile_a$ corresponds to a local $3$-reflective projective billiard $(\pi_A\circ\pi(U),\pi_B\circ\pi(U),\pi_C\circ\pi(U))$. By Theorem \ref{thm:main_theorem_analytic}, the latter is a right-spherical billiard. By analyticity of $S\etoile_a$ and the fact that the latter is connected, $(proj_A\circ\pi(S\etoile_a),proj_B\circ\pi(S\etoile_a),proj_C\circ\pi(S\etoile_a))$ is a right-spherical billiard. Therefore, the $r$-jets of the classical borders of $\mathcal{B}$ at all points are limits of $r$-jets of lines. Hence the classical borders of $\mathcal{B}$ are supported by lines. To conclude, we can state the

\begin{lemma}
If the classical borders $a,b,c$ of $\mathcal{B}$ are supported by lines, the latter is a right-spherical billiard.
\end{lemma}

\begin{proof}
We just have to prove that $\alpha$, $\beta$ $\gamma$ are analytic, because in that case we can apply Proposition \ref{prop:reflectivity_and_lines1} (the same proof works also for real-analytic line-framed curves).

Indeed, let $\ell_a$, $\ell_b$, $\ell_c$ be the three lines supporting $a,b,c$. Fix any point $m_A=(A,L_A)\in\alpha$. Any point $B\in \ell_b$, defines a line $AB$ and the reflected line $AB^{\ast m_A}$ obtained from the projective reflection law at $m_A$. The latter intersect $\ell_c$ at a point $C$. By the implicit function theorem and since $B\mapsto AB^{\ast m_A}$ is analytic, the map $B\mapsto C$ is analytic. Hence the map $B\mapsto BC$ is analytic, and one can define a line $L_B$, analytically depending on $B$, and such that $L_B$ is the only line passing through $B$ for which the quadruple $(BA,BC,\ell_b,L_B)$ is harmonic. Hence we have found an analytic parametrization of $\beta$. The same can be done for $\gamma$ by replacing $L_B$ by $L_C$.
\end{proof}
\end{proof}

\begin{proof}[Proof of Theorem \ref{thm:main_theorem_multidim}]
Here $d\geq 3$. Suppose that one can find $\mathcal{B}=(\alpha, \beta,\gamma)$ a local smooth $3$-pseudo-reflective projective billiard. The idea of the proof is a consequence of \cite{glutkud2}, Theorem 26. This theorem uses the notion of \textit{pseudo-integral surfaces} which can be defined as follows: a pseudo-integral surface $S$ of a Pfaffian system $\mathcal{P}=(M,\mathcal{D},k;(\mathcal{D}_i)_i)$ is the same as an integral surface, except that the inclusion $T_xS\subset\mathcal{D}(x)$ and the transversality conditions only hold for certain $x$ lying in a subset $V\subset S$ of non-zero Lebesgue measure. 

Let $S$ be the set of triples $z=(p,q,r)$ where $p=(A,L_A)\in\alpha$, $q=(B,L_B)\in\beta$, and $r=(C,L_C)\in\gamma$ is defined by the condition that $AB$ and $BC$ are symmetric with respect to the projective reflection law at $q$. We can easily state that $S$ is a smooth pseudo-integral surface of $\mathcal{D}$, noticing that $T_zS\subset\mathcal{D}(z)$ if and only if $z$ is a triangular orbit of $\mathcal{B}$. The first lift $S^{(1)}\subset\gr_{2(d-1)}(M)$ of $S$ doesn't necessarily lie in $M\etoile$: the maps $dproj_G$, $G=A,B,C$, are of course of rank $1$ on the tangent planes of $S$, but not all tangent planes of $S$ lie in $\mathcal{D}$ nor are integral. If we denote by $V\subset S$ the set of points $z\in S$ for which $T_z S\subset\mathcal{D}(z)$, by \cite{glutkud2} Lemma 27 if $z$ is Lebesgue point of $V$ then $T_z S$ is an integral plane of $\mathcal{D}$. Furthermore, if we fix a $p=(z,T_zS)\in S^{(1)}$ where $z$ is a Lebesgue point of $V$, then as in the proof of \cite{glutkud2} Theorem 26 we can project a small neighborhood $U=U(p)$ in $S^{(1)}$ to a stratum $M'$ of $M\etoile$ (by orthogonal projection with respect to some local analytic Riemannian metric) so that the projected surface $S'\subset M'$ is a pseudo-integral surface of the Pfaffian system $\mathcal{P}'=(M',\mathcal{K},2(d-1),\ker d\pi)$. 

Applying \cite{glutkud2} Theorem 26, one can find a sequence of prolongations $\mathcal{P}^{(r)}=(M^{(r)},\ldots)$ of $\mathcal{P}'$ such that $M^{(r)}\neq \emptyset$ for all $r$. Therefore Theorem \ref{theorem:cartan_ku_ra} implies the existence of an analytic integral surface of $\mathcal{P}'$ in $M'\subset M\etoile$. Hence Proposition \ref{prop:reflective_billiard_distribution} implies the existence of a local $3$-reflective analytic projective billiard, which contradicts Theorem \ref{thm:main_theorem_multidim_analytic}.
\end{proof}

\end{document}